\documentclass[12pt]{amsart}

\usepackage{graphics}
\usepackage{graphicx}
\usepackage{latexsym}
\usepackage[english]{babel}
\usepackage{amssymb,amscd,amsmath}
\usepackage[latin1]{inputenc}
\usepackage{a4wide}
\usepackage{fullpage}
\usepackage{pifont}
\usepackage{delarray}
\usepackage{dsfont}

\newtheorem{theorem}{Theorem}[section]

\newtheorem{definition}{Definition}[section]

\newtheorem{Lem}{Lemma}[section]
\newtheorem{Ex}{Example}[section]

\newtheorem{Rem}{Remark}[section]
\title{Rejecting inadmissible rules in reduced normal forms in $S4$}
\author{Mojtaba Aghaei, Maryam Rostami Giv\\ Department of mathematics, Isfahan University of Technology, Isfahan, 84156-83111, Iran \\%%School of Mathematics, Institute for Research in Fundamental Sciences (IPM)
 {\tt aghaei@cc.iut.ac.ir}\\{\tt m.rostamigiv@math.iut.ac.ir}}
\begin{document}
\maketitle
%%%%%%%%%%%%%   Abstract  %%%%%%%%%%%%%%%%%%%%%%%%
\begin{abstract}
Several methods for checking admissibility of rules in the modal logic $S4$ are presented in \cite{Sergey1}, \cite{18}. These methods determine  admissibility of  rules in $S4$, but they don't determine or give substitutions rejecting inadmissible rules. In this paper, we investigate some relations between one of the above methods, based on the reduced normal form rules, and sets of substitutions which reject them. We also generalize the method in \cite{Sergey1}, \cite{18} for one rule to admissibility of a set of rules.

Keywords: sequent calculus, admissible rule, modal logic, S4, substitution, complexity.

2010 Mathematics Subject Classification: 03B47, 03D15.
\end{abstract}

%===============================================

\section{Introduction}

%===============================================
\footnote{This research was in part supported by a grant from IPM (NO)}
Logical admissible rules were studied by Lorenzen \cite{12}, Harrop \cite{7}, Mints \cite{13}.The question wether algorithms exist for recognising admissibility of rules, posed by Friedman \cite{12}, is affirmatively solved by Rybakov \cite{15}, \cite{17} for the modal logic $S4$, for a broad range of propositional modal logics, for example $K4$ and $GL$ in \cite{18}, and by Roziere \cite{14} for IPC using methods of proof theory.

Algorithms deciding admissibility for some transitive modal logics and IPC,
based on projective formulae and unification, are described in Ghilardi \cite{3}, \cite{4},\cite{5},\cite{6}.
They combine resolution and tableau approaches for finding projective approximations of a formula and rely on the existence of an algorithm for theorem proving. Lemhoff and Metcalfe \cite{8} introduced a Gentzen-style system for analytic proof systems that derive admissible rules of non-classical logics.

A practically feasible realisation for S4 built on the algorithm for IPC in \cite{5} is described
in \cite{23}. These algorithms were specifically designed for finding general solutions for
matching and unification problems. In contrast, the original algorithm of \cite{15} can
be used to find only some solution of such problems in S4. In \cite{Sergey1} are presented some methods, specially a tableau method, for checking rule admissibility in $S4$. For more references for inadmissible rules in $S4$ see \cite{1}, \cite{11}, \cite{15}, \cite{18}.

In this paper, in section $2$, we present deduction systems for $S4$ and some results in this system useful for producing substitutions rejecting some inadmissible rules. In section $3$, using Kripke models based on $S4$-formulas in the normal form as their nodes, we provide necessary and sufficient conditions to the determine validity or admissibility of one or several rules (with the same substitution rejecting them) in the normal forms. A way to build the sets used in this conditions and the relations between them and validity or admissibility of rules is presented. In section $4$, the relations between sets of substitutions rejecting sets of rules, and based on them an algorithm to decompose them to their components, are presented. We conclude this section by applying the algorithm on an example. The different ways to decompose the sets make different branches of trees for which examples are presented and show  complexity of the problem of producing substitutions rejecting a set of rules.

\section{Deduction systems for $S4$}
To study substitutions rejecting admissibility of rules in $S4$, specially those reject admissibility of the rules $\cfrac{\diamond p}{p}$ and $\cfrac{\diamond p,p\leftrightarrow \Box p}{p}$, we introduce deduction systems for $S4$ and some useful results in $S4$. A Hilbert system for $S4$ is obtained by adding to the language, axiom schemas and rules for classical logic, the modal $\Box$ and the schemas:
\begin{flushleft}
$K : \Box(A \rightarrow B) \rightarrow (\Box A \rightarrow \Box B)$\\

$T : \Box A \rightarrow A $\\

$4: \Box A \rightarrow \Box\Box A $\\

and the necessitation rule, If $\vdash A$ then $\vdash \Box A$.\\
\end{flushleft}

$\diamond A$ is defined as $\neg\Box\neg A$.

The Gentzen sequent system $G1s$ for $S4$ is \cite{24}:

\vspace{1cm}

$\hspace*{1.1cm}Ax: \hspace*{1cm}  A\vdash A \hspace*{5.7cm} \perp L: \hspace*{1.2cm} \perp\vdash $

\vspace{1cm}

$~~~~~~~~\wedge L:\hspace*{1cm}\frac{A, B, \Gamma\vdash \Delta}{A\wedge B, \Gamma\vdash \Delta}\hspace*{+4.9cm}~~~~\wedge R:\hspace*{1cm}\frac{\Gamma\vdash A, \Delta~~~~~\Gamma\vdash B, \Delta}{\Gamma\vdash A \wedge B, \Delta} $

\vspace{1cm}

$ \hspace*{1.1cm}\vee L:\hspace*{1cm}\frac{A, \Gamma\vdash \Delta ~~~~~~~ B, \Gamma\vdash \Delta}{A \vee B, \Gamma\vdash \Delta}$\hspace*{+4.2cm}$\vee R:\hspace*{1cm}\frac{\Gamma\vdash A, B, \Delta}{\Gamma\vdash A \vee B, \Delta} $

\vspace{1cm}

$\hspace*{0.9cm}\rightarrow L:\hspace*{1cm}\frac{\Gamma\vdash A,\Delta~~~~~~~~~~B,\Gamma\vdash \Delta}{A\rightarrow B,\Gamma\vdash \Delta}\hspace*{+3.5cm}\rightarrow R:\hspace*{0.9cm}\frac{A,\Gamma\vdash B,\Delta}{\Gamma\vdash A\rightarrow B,\Delta}$

\vspace{1cm}

$ \hspace*{1.2cm}\neg L:\hspace*{1cm}\frac{\Gamma\vdash A, \Delta}{\neg A, \Gamma\vdash \Delta}\hspace*{+6cm}\neg R:\hspace*{0.9cm}\frac{A, \Gamma\vdash \Delta}{\Gamma\vdash\neg A,\Delta} $

\vspace{1cm}

$\hspace*{1.1cm}LW:\hspace*{1cm}\frac{\Gamma\vdash\Delta}{A,\Gamma\vdash\Delta}$\hspace*{6cm}$RW:\hspace*{0.9cm}\frac{\Gamma\vdash\Delta}{\Gamma\vdash\Delta,A}$

\vspace{1cm}
\begin{center}
$Cut:\hspace*{1cm}\frac{\Gamma\vdash A,\Delta~~~~~~~~~~A,\Gamma\vdash \Delta}{\Gamma\vdash \Delta} $ \vspace{1cm}
\end{center}
\vspace{1cm}

$\hspace*{1.5cm}\Box L:\hspace*{1cm}\frac{A,~,~\Box A~\Gamma\vdash \Delta}{\Box A,~\Gamma\vdash \Delta}\hspace*{+4.5cm}~~~~~\Box R:\hspace*{1cm}\frac{\Box \Gamma\vdash A,~\diamond\Delta}{\Box\Gamma,~\Gamma^{'}\vdash \Box A,~\diamond\Delta,~\Delta^{'}} $ \vspace{1cm}

\vspace{1cm}

$ \hspace*{1.7cm}\diamond L:\hspace*{1cm}\frac{A,~\Box\Gamma\vdash \diamond\Delta}{\diamond A,~\Box\Gamma,~\Gamma^{'}\vdash \diamond\Delta,~\Delta^{'}}\hspace*{+4.4cm}\diamond R:~~~~~~\frac{\Gamma\vdash A,~\diamond A,~\Delta}{\Gamma\vdash \diamond A,\Delta} $ \vspace{1cm}

In this paper, to produce substitutions rejecting admissibility of rules, we need formula $A$ with $\nvdash A$ and $\vdash \diamond A$, and/or $\vdash A\leftrightarrow \Box A$; See Table $1$ and Table $2$.
\begin{Lem}The following hold in $S4$:

$(1)\vdash\Box (A\rightarrow B)\rightarrow(\diamond A\rightarrow \diamond B)$

$(2)$ If $\vdash A\rightarrow B$ then $\vdash \Box A\rightarrow \Box B$

$(3)$ If $\vdash A\rightarrow B$ then $\vdash \diamond A\rightarrow \diamond B$

$(4)\vdash \Box\Box A\leftrightarrow \Box A$

$(5)\vdash\diamond\diamond A\leftrightarrow \diamond A$

$(6)\vdash\Box\diamond\Box\diamond A\leftrightarrow\Box\diamond A$

$(7)\vdash\diamond\Box\diamond\Box A\leftrightarrow\diamond\Box A$.

$(8)\vdash\Box A\rightarrow \circ A$ where $\circ$ is a sequence of the modals $\Box$ and $\diamond$.
\end{Lem}
\begin{Lem}

The following are provable in $S4$:

$(1)\vdash(\Box A\rightarrow B)\rightarrow \diamond(A\rightarrow B)$

$(2)\vdash\diamond(A\rightarrow \diamond B)\rightarrow(\Box A\rightarrow\diamond B)$

$(3)\vdash(\diamond A\rightarrow \Box B)\leftrightarrow \Box(\diamond A\rightarrow \Box B)$

$(4)\vdash(\diamond A\rightarrow \Box B)\rightarrow \Box(A\rightarrow \Box B)$

$(5)\vdash (A\leftrightarrow\diamond A)\wedge(B\leftrightarrow\Box B)\rightarrow((A\rightarrow B)\leftrightarrow\Box (A\rightarrow B))$

$(6)\vdash(A\leftrightarrow \Box A)\wedge(B\leftrightarrow\Box B)\rightarrow((A\vee B)\leftrightarrow\Box(A\vee B))$

$(7)\vdash(\Box A\vee \Box B)\leftrightarrow \Box(\Box A\vee \Box B)$

$(8)\vdash(\Box A\vee \Box B)\rightarrow \Box(A\vee \Box B)$
\end{Lem}

\begin{proof}

$(1)$ We prove $\vdash\Box\neg(A\rightarrow B)\rightarrow\neg(\Box A\rightarrow B) $

$ 1.\Box\neg(A\rightarrow B)\hspace*{4.5cm}$~~~ assumption

$2.\neg(A\rightarrow B)\rightarrow A$ \hspace{4.1cm} tautology

$3.\Box(\neg(A\rightarrow B))\rightarrow\Box A$ \hspace{3.2cm} by $2,2.1(2)$

$4.\Box A\hspace{6.7cm}1,3, $ MP

$5.\neg(A\rightarrow B)\rightarrow\neg B$ \hspace{3.9cm} tautology

$6.\Box\neg(A\rightarrow B)\rightarrow\Box\neg B$ \hspace*{3.4cm}$5,2.1(2)$

$7.\Box\neg B$\hspace*{6.4cm}$1,6$, MP

$8.\Box\neg B\rightarrow\neg B$\hspace*{5cm} by axiom $T$

$9.\neg B$\hspace*{6.6cm} by $9, 10$, MP

$10.\Box A\rightarrow(\neg B\rightarrow\neg(\Box A\rightarrow B))$ \hspace*{1.9cm}tautology

$11.\neg(\Box A\rightarrow B)$ \hspace*{4.8cm}$4, 9,10$, MP\\

$(2)$ We prove $ \vdash\neg(\Box A\rightarrow \diamond B)\rightarrow\Box\neg(A\rightarrow \diamond B) $

$1.\neg(\Box A\rightarrow \diamond B)$ \hspace*{3.9cm} assumption

$2.\neg(\square A\rightarrow \diamond B)\rightarrow\square A$ \hspace*{2.8cm}tautology

$3.\square A$\hspace*{6cm}by $1,2,MP $

$4.\neg(\square A\rightarrow \diamond B)\rightarrow\neg \diamond B$\hspace*{2.5cm}tautology

$5.\neg \diamond B$ \hspace*{5.5cm}by $1, 4, MP $

$6.\Box\neg B$ \hspace{5.6cm}by $5$

$7.\Box\Box\neg B$ \hspace*{5.3cm}by $6$ and by axiom $4$

$8.\Box\neg \diamond B$\hspace*{5.8cm}by $7 $

$9.A\rightarrow(\neg\diamond B \rightarrow \neg (A \rightarrow \diamond B))$ \hspace*{1.8cm} tautology

$10.\Box A \rightarrow \Box (\neg \diamond B \rightarrow \neg (A \rightarrow \diamond B)) $ \hspace*{1cm} by $9,2.1(2)$

$11.\Box(\neg \diamond B \rightarrow \neg(A \rightarrow \diamond B))$ \hspace*{2.4cm}by $10, 3, MP$

$12.\Box \neg \diamond B \rightarrow \square \neg (A \rightarrow \diamond B)$ \hspace*{2.3cm} by $11$ and axiom $K$

$13.\Box\neg(A\rightarrow \diamond B)$ \hspace*{4.5cm}by $8, 12, MP $\\

$(3)$ We get $\vdash \neg \Box (\diamond A\rightarrow \Box B)\rightarrow \neg (\diamond A\rightarrow\Box B)$

$1.\neg \Box (\diamond A\rightarrow \Box B)$\hspace{4.5cm} assumption

$2.\diamond\neg(\diamond A\rightarrow \Box B)$ \hspace{4.3cm} by $1$

$3.\neg(\diamond A\rightarrow \Box B)\rightarrow \diamond A$ \hspace{3.6cm} tautology

$4.\diamond\neg(\diamond A\rightarrow\Box B)\rightarrow\diamond\diamond A$ \hspace{2.8cm} by $3,2.1(3)$

$5.\diamond\diamond A$ \hspace{6.3cm} by $2$ and $4$ and MP

$6.\diamond\diamond A\rightarrow\diamond A$ \hspace{5.1cm} by axiom $4$

$7.\diamond A$ \hspace{6.5cm} by $5$ and $6$ and MP

$8.\neg(\diamond A\rightarrow \Box B)\rightarrow \neg\Box B$ \hspace{3.2cm} tautology

$9.\diamond\neg(\diamond A\rightarrow\Box B)\rightarrow\diamond\neg\Box B $ \hspace{2.6cm} by $8,2.1(3)$

$10.\diamond\neg\Box B $ \hspace{5.7cm} by $2$ and $9$ and MP

$11.\neg \Box \Box B$ \hspace{5.8cm} by $10$

$12.\neg \Box \Box B\rightarrow\neg\Box B $ \hspace{4.2cm} by axiom $4$

$13.\neg \Box B $ \hspace{6.1cm} by $11$ and $12$ and MP

$14.\diamond A\rightarrow(\neg \Box B \rightarrow\neg(\diamond A\rightarrow \Box B))$ \hspace{1.5cm}tautology

$15.\neg(\diamond A\rightarrow \Box B)$ \hspace{4.7cm} by $7,13,14$\\

$(4)$

$1.(A\rightarrow\diamond A)\rightarrow ((\diamond A \rightarrow \Box B)\rightarrow(A \rightarrow \Box B))$ \hspace{2cm} tautology

$2.\Box(A\rightarrow\diamond A)\rightarrow \Box((\diamond A \rightarrow \Box B)\rightarrow(A \rightarrow \Box B))$ \hspace{1.4cm} by $1,2.1(2)$

$3.A \rightarrow\diamond A$ \hspace{8.1cm} by axiom $4$

$4.\Box(A \rightarrow\diamond A)$ \hspace{7.5cm} by $3$ and necessitation rule

$5.\Box((\diamond A \rightarrow \Box B)\rightarrow(A \rightarrow \Box B))$ \hspace{4.3cm} by $2$ and $4$ and MP

$6.\Box(\diamond A \rightarrow \Box B)\rightarrow \Box(A \rightarrow \Box B)$ \hspace{4.3cm} by $5$ and axiom $K$

$7.(\diamond A\rightarrow \Box B)\leftrightarrow \Box(\diamond A\rightarrow \Box B)$ \hspace{4.4cm} by part $(3)$

$8.(\diamond A\rightarrow \Box B)\rightarrow \Box(A\rightarrow \Box B)$ \hspace{4.6cm} by $6$ and $7$
\end{proof}

For formula $A$ in $S4$, let the property $(*)$ be $\nvdash A$ and $\vdash \diamond A$ and the property $(**)$ be the property $(*)$ in addition $\vdash A\leftrightarrow \Box A$.

Let $\sigma(p)=A$. If $A$ has the property $(*)$ then $\sigma$ rejects the rule $r=\cfrac{\diamond p}{p}$ and if $A$ has the property $(**)$ then $\sigma$ rejects the rule $r=\cfrac{\diamond p, p\leftrightarrow \Box p}{p}$.

\begin{Lem}
Let $A$ and $B$ formulas in $S4$ with the property $(**)$ then $A\vee B$ has the property $(**)$.
\end{Lem}

\begin{proof}
Let $M_1\nvDash A, M_2\nvDash B $ and $\vdash A \vee B$. Let $M=(M_1+M_2)'$ obtained by adding a new root $0$ below $M_1$ and $M_2$. Then $M,0\vDash A\vee B$. Let $M,0\vDash A$, since $\vdash A\leftrightarrow \Box A$ then $M,0\vDash \Box A$ and then $M_1\vDash A$, contradiction: Then $\nvdash A \vee B$. By $\vdash\diamond A \vee \diamond B \leftrightarrow\diamond (A\vee B)$ we get $\vdash\diamond (A\vee B)$, and $\vdash A\vee B\leftrightarrow \Box(A\vee B)$ by Lemma $2.2(6)$.
\end{proof}

We can get formulas $A$ with $\vdash \diamond A$ as follows. By Lemma $2.1(8)$, Lemma $2.2(1)$ and $2.2(4)$, and replacing $p$ with $\Box p$ or $\diamond p$ and simplifying modals by Lemma $2.1$, we get formulas $A$ with $\vdash \diamond A$ as in Lemma $2.4$. By some replacements we get $\vdash A$ and we remove them from the list. Note that if $\vdash\diamond A$ then $\vdash \diamond (A\vee B)$ and $\vdash \diamond (B\rightarrow A)$ and if $B$ is obtained from $A$ by replacement then $\vdash \diamond B$.

\begin{Lem}\label{2001}$\vdash \diamond A$ for the following formulas $A$ in $S4$

$(1)(p \rightarrow \Box p)$\hspace{1.6cm}$(7)\diamond p \rightarrow \diamond\Box\diamond p$\hspace{1.7cm}$(13)\Box(\diamond p\rightarrow \Box \diamond p)$\\

$(2)(p \rightarrow \Box\diamond p)$\hspace{1.2cm}$(8)\diamond \Box p \rightarrow \Box \diamond \Box p$\hspace{1.3cm}$(14)\Box(\diamond p\rightarrow \diamond\Box \diamond p)$\\

$(3)(p \rightarrow \diamond\Box p)$\hspace{1.4cm}$(9)\diamond\Box\diamond p\rightarrow\Box\diamond p$\hspace{1.2cm}$(15)\Box(\diamond \Box p\rightarrow \Box \diamond \Box p)$\\

$(4)(p \rightarrow \Box\diamond\Box p)$\hspace{.9cm}$(10)\Box(p \rightarrow \Box\diamond p)$\hspace{1.5cm}$(16)\Box(\diamond \Box \diamond p\rightarrow \Box \diamond p)$\\

$(5)(p \rightarrow \diamond\Box\diamond p)$\hspace{1cm}$(11)\Box(p \rightarrow \diamond\Box\diamond p)$\hspace{1.4cm}$(17)\Box \diamond (p \rightarrow \Box \diamond \Box p)$\\

$(6)\diamond p \rightarrow \Box \diamond p$\hspace{1.1cm}$(12)\Box(\Box p\rightarrow \Box \diamond \Box p)$\hspace{1.1cm}$(18)\Box \diamond (\diamond p \rightarrow \Box \diamond p)$\\
\end{Lem}
\begin{proof}
$(1)-(5)$ by Lemma $2.1(8)$ and $2.2(1)$, $(6)$ and $(7)$ from $(1)-(5)$, $(8)$ from $(6)$, $(9)$ from $(8)$ by replacement, $(10)-(12)$ from $(6)-(8)$ by Lemma $2.2(1)$, $(13)-(14)$ from $(10)-(12), ~ (15)$ from$(13),~ (16)$ form $(15)$ and $(18)$ from $(17)$ by replacement, and  $(17)$ from $(12)$ by Lemma $2.2(1)$.
\end{proof}

\begin{Rem}
By Lemma $2.2$, except for $(1)-(5)$ and $(7)$ in Lemma $2.4$, we get $\vdash A\leftrightarrow \Box A$. In the lemma note that $(12)$ and then $(17)$ and $(18)$ are theorems in $S4$. The other parts of the lemma are not theorems. Consider $S4$-models $M_1$ and $M_2$. $(1)-(7), (10),(11),(13),(14)$ are not valid in $M_1$; $(8),(9),(15),(16)$ are not valid in $M_2$.

\begin{figure}[ht]
\centering
\includegraphics [scale=.5]{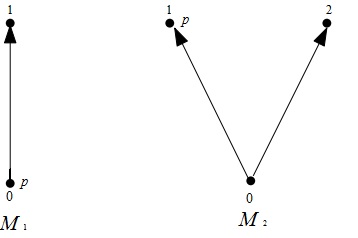}
\caption{\label{fig:math}}
\end{figure}
\end{Rem}

\begin{Ex} The proofs of $\vdash \diamond A$ in Lemma $2.4$ in the sequent calculus are interesting. For example we prove some cases.
\begin{center}

$\cfrac{~~~\cfrac{~~\cfrac{~~~~~\cfrac{~~~~\cfrac{~~~\cfrac{~~p\vdash p}{p\vdash p,\Box p}RW}{\vdash p, p\rightarrow\Box p}\rightarrow R}{\vdash p,\diamond(p\rightarrow\Box p)}\diamond R,RW}{p\vdash\Box p,\diamond(p\rightarrow\Box p)}\Box R}{\vdash p\rightarrow\Box p,\diamond(p\rightarrow\Box p)}\rightarrow R}{\vdash \diamond(p\rightarrow\Box p)}\diamond R$

\vspace{1cm}

$\cfrac{~~\cfrac{~~\cfrac{~~\cfrac{~~~\cfrac{~~~~\cfrac{~~~\cfrac{~~~\cfrac{~~~\cfrac{~~p\vdash p}{p\vdash \diamond p}\diamond R,RW}{p\vdash \Box\diamond p,\diamond p}LW}{\vdash\diamond p,p\rightarrow\Box\diamond p}\rightarrow R}{\vdash\diamond p,\Box(p\rightarrow\Box\diamond p),\diamond\Box(p\rightarrow\Box\diamond p)}\Box R,RW}{\vdash\diamond p,\diamond\Box(p\rightarrow\Box\diamond p)}\diamond R}{p\vdash \Box\diamond p,\diamond\Box(p\rightarrow\Box\diamond p)}\Box R}{\vdash p\rightarrow\Box\diamond p,\diamond\Box(p\rightarrow\Box\diamond p)}\rightarrow R}{\vdash\Box(p\rightarrow\Box\diamond p),\diamond\Box(p\rightarrow\Box\diamond p)}\Box R}{\vdash\diamond\Box(p\rightarrow\Box\diamond p)}\diamond R$

\vspace{1cm}

$\cfrac{\cfrac{\cfrac{\cfrac{\cfrac{\cfrac{\cfrac{\cfrac{\cfrac{\cfrac{\cfrac{\cfrac{~~~p\vdash p~~~~}{p\vdash p,\Box\diamond\Box p}RW}{\vdash p, p\rightarrow\Box\diamond\Box p}\rightarrow R}{\vdash p,\diamond(p\rightarrow\Box\diamond\Box p)}\diamond R, RW}{\vdash\Box p, \diamond(p\rightarrow\Box\diamond\Box p)}\Box R}{\vdash\diamond\Box p, \diamond(p\rightarrow\Box\diamond\Box p)}\diamond R}{\vdash\diamond\Box p, \Box\diamond(p\rightarrow\Box\diamond\Box p)}\Box R}{\vdash\diamond\Box p, \diamond\Box\diamond(p\rightarrow\Box\diamond\Box p)}\diamond R,RW}{p\vdash \Box\diamond\Box p, \diamond\Box\diamond(p\rightarrow\Box\diamond\Box p)}\Box R,LW}{\vdash(p\rightarrow\Box\diamond\Box p), \diamond\Box\diamond(p\rightarrow\Box\diamond\Box p)}\rightarrow R}{\vdash \diamond(p\rightarrow\Box\diamond\Box p), \diamond\Box\diamond(p\rightarrow\Box\diamond\Box p)}\diamond R,RW}{\vdash \Box\diamond(p\rightarrow\Box\diamond\Box p), \diamond\Box\diamond(p\rightarrow\Box\diamond\Box p)}\Box R}{\vdash \diamond\Box\diamond(p\rightarrow\Box\diamond\Box p)}\diamond R$
\end{center}

\end{Ex}

\section{Kripke models and reduce normal forms}
A Kripke frame is a pair $(W,R)$ where $W$ is a non-empty set of words and $R$ is a relation on $W$, i.e. $R \subseteq W × W$. A Kripke model is a triple $(W,R,V)$, where $(W,R)$ is a frame and $V$ is a valuation that assigns sets of worlds to propositional variables, i.e. $V : P\rightarrow P(W)$, where $P$ is the set of propositional variables. $V (p)$ is interpreted as the set of worlds where $p$ is true. An $S4$-model is a Kripke model $(W,R,V)$ in which $R$ is reflexive and transitive.

$M, w\models \psi$ is defined as usual for an $S4$-formula $\psi$, an $S4$-model $M$ and a $w$ in $M$. $\psi$ is valid in a model $M$, denoted $M\models \psi$, if $M, w\models \psi$ for all worlds $w$ in $M$.\\

A rule of inference is as the following: $r=\frac{\alpha_1,...,\alpha_n}{\beta}$ that $\alpha_1,...,\alpha_n$ and $\beta$ are $S4$-formulae. $r$ is valid in an $S4$-model $M$, if from $M\models \alpha_1,...,M\models \alpha_n$ follows $M\models \beta$. The rule $r$ is valid if it is valid in each $S4$-model $M$.

For a rule of inference $r=\frac{\alpha_1,...,\alpha_n}{\beta}$, a logic $\lambda$ and substitution $\sigma$ we use $\lambda\vdash_{\sigma}r$ if from $\lambda\vdash \sigma(\alpha_1),...,\lambda\vdash \sigma(\alpha_n)$ it follows that $\lambda\vdash \sigma(\beta)$. If $\lambda\nvdash_{\sigma}r $ we say $\sigma$ rejects $r$ in $\lambda$. $r$ is admissible in $\lambda$ if $\lambda\vdash_{\sigma}r$ for every substitution $\sigma$.

\begin{definition}\label{201}
A  Kripke  model  $K_n$  is  called $n$-characterizing  for  a  modal logic $\lambda$ (any  normal modal logic, not necessarily an extension of the  system $K4$) if the domain of the valuation $V$ from $K_n$ is the set $P$ which consists of $n$ different propositional variables, and if the following holds. For any formula $\alpha$ which is build up of variables  from $P$
\begin{center}
$\alpha\in \lambda\Leftrightarrow K_n\models\alpha$
\end{center}

We use $Ch_{\lambda}(n)$ for n-characterizing Kripke models $K_n$.
\end{definition}

\begin{Lem}\label{202}
Let $\mathcal{M}=(M, V)$ be a $\lambda$-model.\\

$(1)$ If $\sigma$ is a substitution and $S(x_i)=V(\sigma(x_i))$, then $S$ is a definable valuation for which $S(\alpha)=V(\sigma(\alpha))$ for each $\lambda$ formula $\alpha$, that is $(M,S),w\vDash \alpha$ iff $(M,V),w\vDash \sigma(\alpha)$ for each $w\in M$.\\

$(2)$ If $S$ is a definable valuation in $M$, that is for each variable $x_i$ there is a formula $\phi_i$ such that $(M,S),w\vDash x_i$ iff $(M,V),w\vDash \phi_i$, and if $\sigma(x_i)=\phi_i$ is a substitution then $S(\alpha)=V(\sigma(\alpha))$ for each $\lambda$-formula $\alpha$, that is $(M,S),w\vDash \alpha$ iff $(M,V),w\vDash \sigma(\alpha)$ for each $w\in M$.\\

$(3)$ If $\sigma$ is a substitution, $S$ is a definable valuation for which $S(\alpha)=V(\sigma(\alpha))$ for each $\lambda$ formula $\alpha$, $r:=\frac{\alpha_1,...,\alpha_m}{\beta}$ is a rule and $\sigma(r):=\frac{\sigma(\alpha_1),...,\sigma(\alpha_m)}{\sigma(\beta)}$ then $r$ is valid in $(M,S)$ iff $\sigma(r)$ is valid in $(M,V)$.
\end{Lem}
\begin{proof}
$(1)$ and $(2)$ by easy induction on $\alpha$ and $(3)$ by $(1)$ and $(2)$.\\
\end{proof}
\begin{theorem}\label{203}
Let $(K_n,V_n)$, $n\in N$, be a sequence of $n$-characterizing models for a modal logic $\lambda$. Inference rules $r_1:=\frac{\alpha_{11},...,\alpha_{1m_1}}{\beta_1},...,r_k:=\frac{\alpha_{k1},...,\alpha_{km_k}}{\beta_k}$ are  inadmissible in $\lambda$ with the same substitution $\sigma$ iff $r_1,...,r_k$ are invalid in $(K_n,S)$ for some $n\in N$ and some definable valuation $S$ of variables from $r_1,...,r_k$ in $K_n$ (that is,  If $S(\alpha_{ij})=K_n$ and $S(\beta_i)\neq K_n$ for $i=1,...,k$ and $j=1,...,m_i$).\\
\end{theorem}
\begin{proof}
$(\Leftarrow)$ If $r_1,...,r_k$ are not admissible in $\lambda$ with the same substitution $\sigma$ then for $i=1,...,k$ and $j=1,...,m_i$, we get $\displaystyle\vdash_{\lambda}\sigma(\alpha_{ij})$ and $\displaystyle\nvdash_{\lambda}\sigma(\beta_i)$ then $(K_n,V_n)\displaystyle\vDash_{\lambda}\sigma(\alpha_{ij})$ and  $(K_n,V_n)\displaystyle\nvDash_{\lambda}\sigma(\beta_i)$ for $n$ the number of variables in $\sigma(r_1),...,\sigma(r_k)$, thus $\sigma(r_1),...,\sigma(r_k)$ is invalid in $(K_n,V_n)$, so $r_1,...,r_k$ are invalid in $(K_n,S)$ with the definable valuation $S(x_i)=V_n(\sigma(x_i))$ for each variable $x_i$ free in $r$ by Lemma 3.1.\\
$(\Rightarrow)$ If $r_1,...,r_k$ is invalid in $(K_n,S)$ with the definable valuation $S$ for each variable $x_i$ free in $r_1,...,r_k$ then  by Lemma $3.1$, there is a substitution $\sigma$ for each free variable $x_i$ in $r_1,...,r_k$ such that $S(x_i)=V_n(\sigma(x_i))$ and $\sigma(r_1),...,\sigma(r_k)$ are invalid in $(K_n, V_n)$, that is, for $i=1,...,k$ and $j=1,...,m_i$ we get $(K_n,V_n)\displaystyle\vDash_{\lambda}\sigma(\alpha_{ij})$ and  $(K_n,V_n)\displaystyle\nvDash_{\lambda}\sigma(\beta_i)$, then $\displaystyle\vdash_{\lambda}\sigma(\alpha_{ij})$ and  $\displaystyle\nvdash_{\lambda}\sigma(\beta_i)$, therefore $r_1,...,r_k$ are not admissible in $\lambda$.\\
\end{proof}

\begin{definition}\label{204}A rule $r$ is said to be in reduced normal form if it has the form
\begin{center}
$(rnf)~~~~~~~~~~~~~r=\frac{\displaystyle\bigvee_{j\in I}\phi_j}{\displaystyle\bigvee_{j\in J}\phi_j}$ ~~~~~or ~~~~~$r=\frac{\displaystyle\bigvee_{j\in I}\phi_j}{\displaystyle p_i}$
 \end{center}
and each disjunct $\phi_j$ has the form
\begin{center}
$\displaystyle\bigwedge_{0\leq i\leq n} p_i^{t(i,j,0)}\wedge\displaystyle\bigwedge _{0\leq i\leq n} \diamond p_i^{t(i,j,1)} $,
\end{center}
where ($i$) all $\phi_j$  are different $(ii)~p_0,...,p_n$ denote propositional variables, $(iii)$ $t$ is a Boolean function $t:\{0,...,n\}\times\{1,...,s\}\times\{0,1\}\rightarrow\{0,1\}$ and ($iv$) $\alpha^{0}=\neg \alpha$ and $\alpha^{1}=\alpha$ for any formula $\alpha$.
\end{definition}

\begin{Rem}
For the rule $r=\frac{\bigvee_{i\in I}\phi_i}{\bigvee_{j\in J}\phi_j}$ in reduced normal form we can suppose $J\subseteq I$, because let  $r'=\frac{\bigvee_{i\in I}\phi_i}{\bigvee_{j\in J-\{k\}}\phi_j}, k\in I-J,$ then $r$ and $r'$ are rejected by the same substitutions. Let a substitution $\sigma$ rejects $r'$, i.e., $\vdash \sigma(\bigvee_{i\in I}\phi_i)$ and $\nvdash \sigma(\bigvee_{j\in J-\{k\}}\phi_j)$. For each $i\in I,~~\phi_i\rightarrow \neg \phi_k$ and then $\bigvee_{i\in I}\phi_i \rightarrow \neg \phi_k$ are tautologies, so $\vdash \neg \sigma(\phi_k)$ and therefore $\nvdash \sigma(\bigvee_{j\in J}\phi_j)$. Then $\sigma $ rejects $r$. The other side is obvious. For simplicity, we use $r=\frac{I}{J}$
\end{Rem}
Using the renaming technique any modal rule can be transformed into an equivalent rule in reduced normal form \cite{Sergey1}, \cite{18}.
\begin{theorem}\label{205}
Any rule $r=\frac{\alpha}{\beta}$ can be transformed in exponential time to an equivalent rule in reduced normal form.\\
\end{theorem}
\begin{proof}
see \cite{Sergey1}, \cite{18}.\\
\end{proof}

 Let $\Theta_{n}=\{\phi_1,...,\phi_s\}$ be the
set of all disjuncts in $n$ variables $p_1,...,p_n$. For every $\phi_j\in \Theta_n, let$
\begin{center}
$\theta(\phi_j)=\{p_i~|~t(i,j,0)=1\}\hspace{2cm}\theta_{\diamond}(\phi_j)=\{p_i~|~t(i,j,1)=1\}$.
\end{center}
For every subset of disjuncts $W\subseteq\Theta_n$, let $M(\Theta_n, W)$
denote the Kripke model in which

\begin{center}
$W^{M}=W$, $R^{M}=\{(\phi_1,\phi_2)~|~\theta_\diamond(\phi_2)\subseteq \theta_\diamond(\phi_1)\}$, $(p_i,\phi_j)\in\nu^{m}\Leftrightarrow p_i\in \theta(\phi_j)$.
\end{center}
\begin{Ex}
In Figure $2$ and Figure $3$ we show $M(\Theta_2,\Theta_2)$ and $M(\Theta_3,\Theta_3)$, in which $i$ is used for $\phi_i$ for which $\Theta(\phi_i)$ and $\Theta_{\diamond}(\phi_i)$ are determined as they are showed in the figures. For example in Figure $2$:

$\phi_1:=p_1\wedge p_2\wedge \diamond p_1\wedge \diamond p_2$, $\phi_6:=\neg p_1\wedge p_2\wedge \diamond p_1\wedge \neg\diamond p_2$, $\phi_{16}:=\neg p_1\wedge \neg p_2\wedge \neg\diamond p_1\wedge \neg\diamond p_2$.\\

If $\Theta_{\diamond}(\phi_i)\subseteq \Theta(\phi_i)$ we show the node $i$ by $\bullet$ and otherwise we show it by $\circ$ for which $\vdash \neg \phi_i$ and then we can remove it from disjunctions in the rules. The right figures only show the nodes $\bullet$.

The first column in Table $1$ and Table $2$ show all the sets $W$, as in Theorem $3.5$, for which if $W\subseteq I$, the rule $r=\cfrac{I}{p_1}$ is inadmissible. The second column shows the simplified form of $\displaystyle\bigvee_{i\in W}\phi_{i}$, and the third shows the conditions on formulas $A, B$ and $C$ for them the substitution $\sigma$ with $\sigma(p_1)=A, \sigma(p_2)=B$ and $\sigma(p_3)=C$ rejects the rule $r$. We suppose in the tables $\vdash\diamond A$ except if $\neg\diamond p_1$ occurs in the simplified form. In the last three cases in Table $2$, the substitution $\sigma$ has not been found yet.\\
\begin{figure}
\centering
\includegraphics [scale=.5]{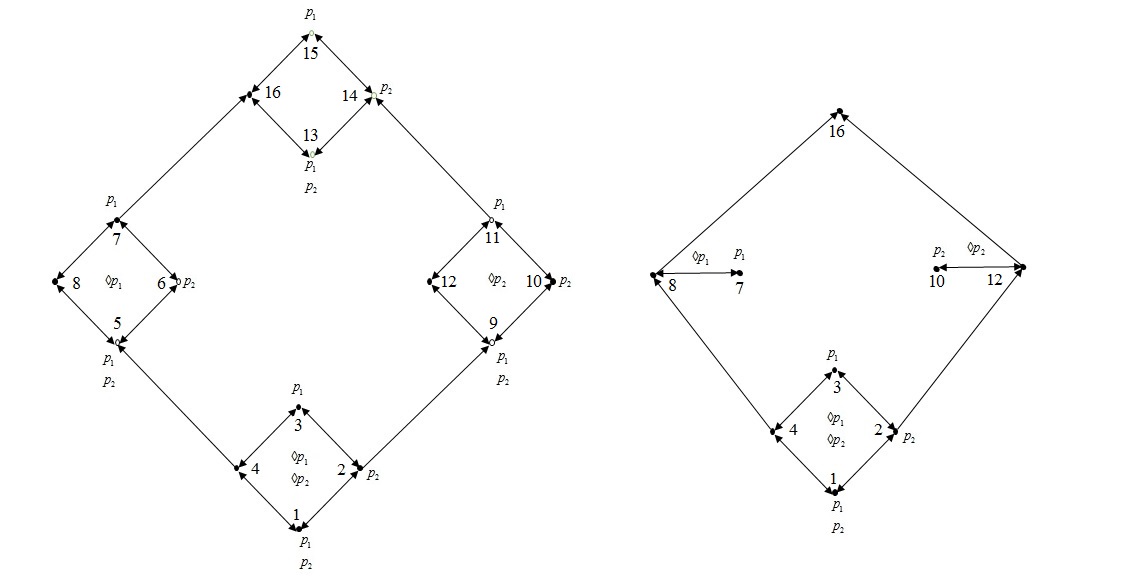}
\caption{\label{fig:math}$M(\Theta_2,\Theta_2)$}
\end{figure}
\begin{table}
\centering
\includegraphics [scale=.6]{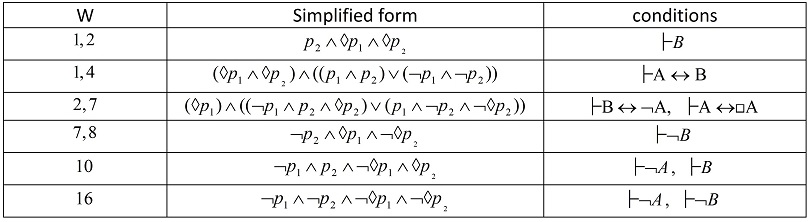}
\caption{\label{fig:2variables}}
\end{table}
{\newpage
\begin{figure}
\centering
\includegraphics [scale=.4]{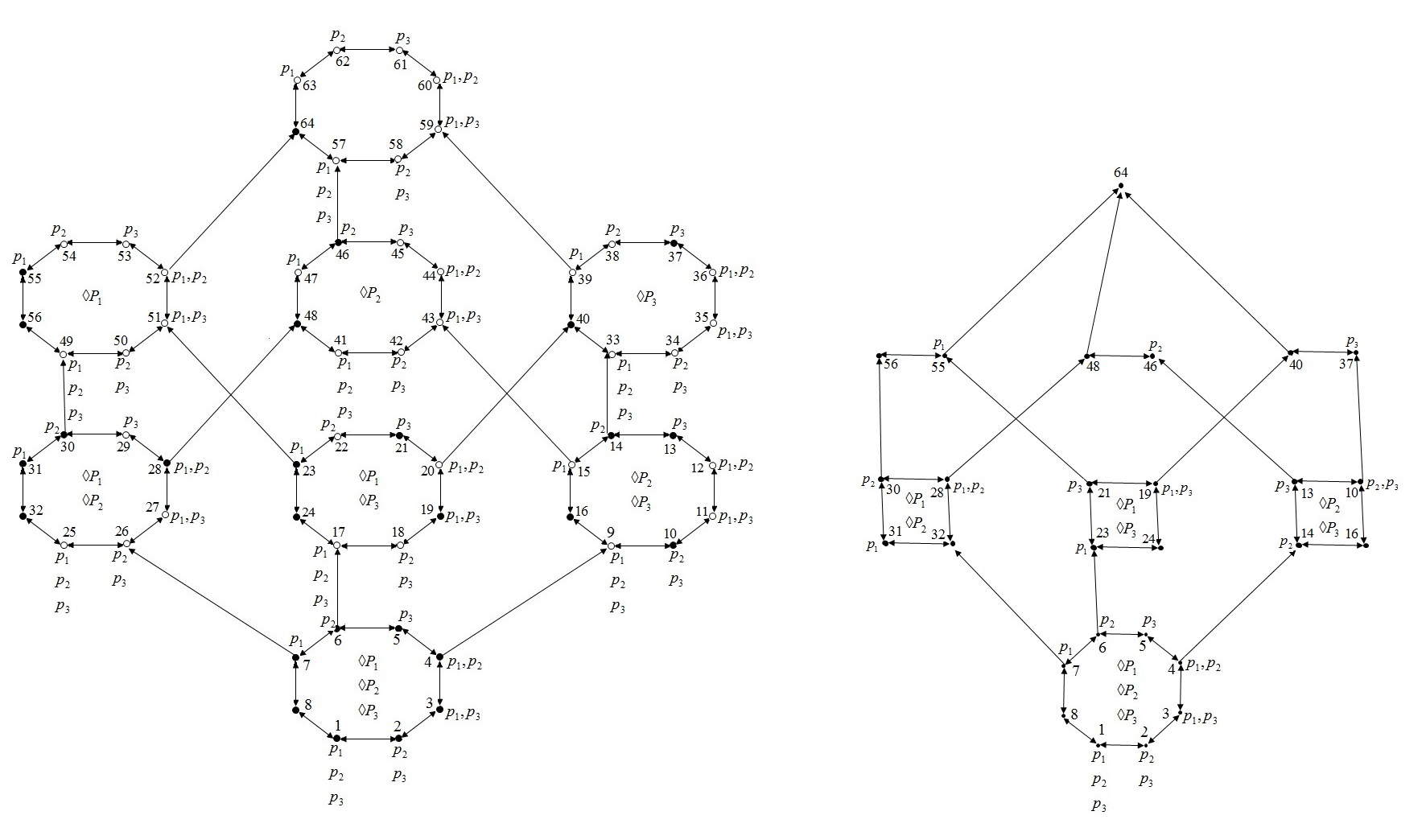}
\caption{\label{fig:nath}$M(\Theta_3,\Theta_3)$}
\end{figure}
\begin{table}
\centering
\includegraphics [scale=.6]{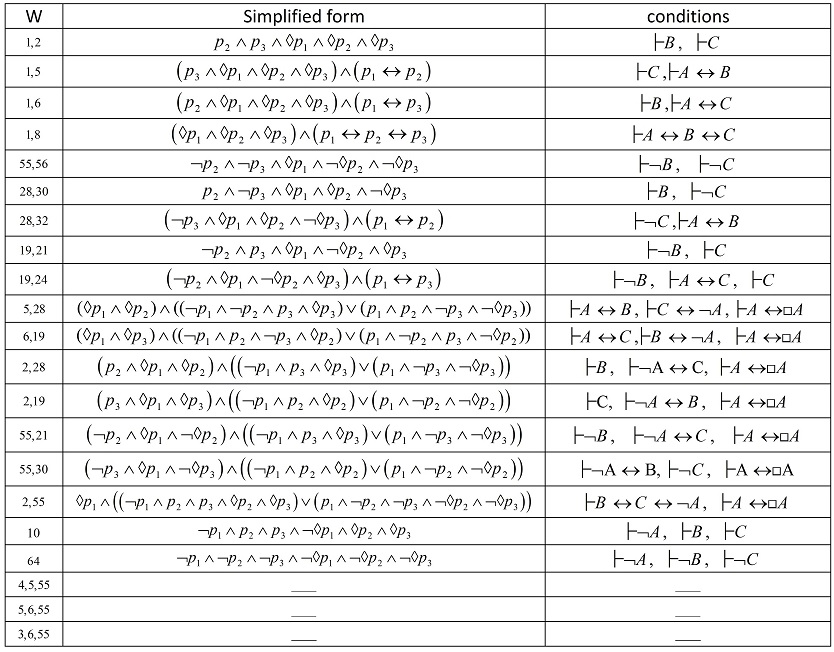}
\caption{\label{fig:3variables}}
\end{table}
}
\end{Ex}
\begin{theorem}\label{206}
Let ${\mathcal N}=(N,R)$ and ${\mathcal N}\models \displaystyle \bigvee_{i\in I}\phi_i$ and $W=\{\phi_i\in \Theta_n\mid \exists x\in N~~{\mathcal N},x\models \phi_i\}$. Then

$(1)$ If ${\mathcal N},x\models \phi_i$ then ${\mathcal N},x\models \phi$ iff ${\mathcal M}(\Theta_n,W),\phi_i\models \phi$ for each formula $\phi$.\\

$(2)$~$W\subseteq\{\phi_i\in \Theta_n\mid i\in I, {\mathcal M}(\Theta_n,W),\phi_i\models \phi_i\}$.\\

$(3)$~${\mathcal M}(\Theta_n,W)\models \displaystyle\bigvee_{i\in I'}\phi_i$ iff $W\subseteq \{\phi_i\in\Theta_n\mid i\in I'\}$.\\

$(4)$~${\mathcal N}\vDash \displaystyle\bigvee_{i\in I'}\phi_i$ iff $W\subseteq \{\phi_i\in\Theta_n\mid i\in I'\}$.\\

$(5)$ If ${\mathcal N}\nvDash p_i $ iff ${\mathcal M}(\Theta_n,W)\nvDash p_i$ for $i=1,...,n$.\\

$(6)$ If for each subset $D$ of $ N$ there exists $a\in N$ such that
\begin{center}
$\theta_{\diamond}(a)=\theta(a)\cup\displaystyle\bigcup_{d\in D}\theta_{\diamond}(d)$
\end{center}
then for each subset $D$ of $W$ there exists $\phi_j\in W$ such that
\begin{center}
$\theta_{\diamond}(\phi_j)=\theta(\phi_j)\cup\displaystyle\bigcup_{\phi\in D}\theta_{\diamond}(\phi)$
\end{center}
where $\Theta(a)=\{p_i\mid N,a\vDash p_i\}$ and $\Theta_{\diamond}(a)=\{p_i\mid N,a\vDash \diamond p_i\}$.\\
\end{theorem}

\begin{proof}

First note that if ${\mathcal N},x\models \phi_i$ and ${\mathcal N},x\models \phi_j$ then $\phi_i=\phi_j$. Define $f:N\displaystyle\longrightarrow_{onto} W$ with $f(x)=\phi_i$ iff ${\mathcal N},x\models \phi_i$. $f$ is a homomorphism, because for $x,y \in N$ if $xRy$, $f(x)=\phi_i, f(y)=\phi_j $ and $p_k\in \Theta_{\diamond}(f(y))=\Theta_{\diamond}(\phi_j)$, since ${\mathcal N},y\models \phi_j$ then ${\mathcal N},y\models \diamond p_k$, thus ${\mathcal N},x\models \diamond p_k$ and since ${\mathcal N},x\models\phi_i$, so $ p_k\in \Theta_{\diamond}(f(x))=\Theta_{\diamond}(\phi_i)$. Therefore $\Theta_{\diamond}(f(y))\subseteq \Theta_{\diamond}(f(x))$, then $f(x)R^{\mathcal M} f(y)$.\\

$(1)$ Since $f$ is a surjective homomorphism then ${\mathcal N},x\models \phi$ iff ${\mathcal M}(\Theta_n,W),f(x)\models \phi$ iff ${\mathcal M}(\Theta_n,W),\phi_i\models \phi$, if ${\mathcal N},x\models \phi_i$.\\

$(2)$~By $(1)$ for $\phi=\phi_i\in W$ let ${\mathcal N},x\models \phi_i$ then ${\mathcal M}(\Theta_n,W),\phi_i\models \phi_i$. Again $W\subseteq\{\phi_i\mid i\in I\}$; Because, if $j\notin I$ and $\phi_j\in W $, let ${\mathcal N},x\vDash \phi_j$, since $\phi_j\rightarrow \neg\phi_i$ for each $i\in I$ and thus $\phi_j\rightarrow \neg\displaystyle\bigvee_{i\in I}\phi_i $ are tautologies  then ${\mathcal N},x\vDash\neg\displaystyle\bigvee_{i\in I}\phi_i $, contradicting ${\mathcal N}\vDash\displaystyle\bigvee_{i\in I}\phi_i $.\\

$(3)(\Rightarrow)$ Let $\phi_j\in W$and $j\in I' $. Since ${\mathcal M}(\Theta_n,W),\phi_j\models \phi_j$ and $\models \phi_j\rightarrow \neg \phi_i$ for $i\neq j$ and then $\models \phi_j\rightarrow \neg\displaystyle\bigvee_{i\in I'}\phi_i$ so ${\mathcal M}(\Theta_n,W),\phi_j\models \neg\displaystyle\bigvee_{i\in I'}\phi_i$.\\

$(\Leftarrow)$ By $(2)$.\\

$(4)$ By $(1)$ and $(3)$.\\

$(5)$ Let ${\mathcal N},w\nvDash p_i$ and $f(w)=\phi_j$ then ${\mathcal M}(\Theta_n,W),f(w)\nvDash p_i$ therefore ${\mathcal M}(\Theta_n,W),\phi_j\nvDash p_i$.\\

$(6)$ Let $D\subseteq W$. Since $f$ is surjective, for each $\phi_i\in D$ pick  and  fix a  representative $d_i$ such that $f(d_i)=\phi_i$ and take $D'=\{d_i\in N \mid \phi_i\in D\}$. Let $a\in N, f(a)=\phi_j$  and $\theta_{\diamond}(a)=\theta(a)\cup\displaystyle\bigcup_{d\in D'}\theta_{\diamond}(d)=\theta(a)\cup\displaystyle\bigcup_{\phi_i\in D}\theta_{\diamond}(d_i)$, then by $\Theta(a)=\Theta(f(a))=\Theta(\phi_j)$ and $\Theta_{\diamond}(a)=\Theta_{\diamond}(f(a))=\Theta_{\diamond}(\phi_j)$ and $\Theta_{\diamond}(d_i)=\Theta_{\diamond}(f(d_i))=\theta_{\diamond}(\phi_i))$, we get $\Theta_{\diamond}(\phi_j)=\theta(\phi_j)\cup\displaystyle\bigcup_{\phi_i\in D}\theta_{\diamond}(\phi_i))$.\\

\end{proof}

\begin{theorem}\label{207}

A rule $r=\frac{\displaystyle\bigvee_{i\in I}\phi_i\vee\displaystyle\bigvee_{j\in J}\phi_j}{\displaystyle\bigvee_{j\in J}\phi_j}$ is invalid for $S4$ models iff there is a set $W\subseteq \{\phi_i\in\Theta_n\mid i\in I\cup J\}$ such that

$(1)$~${\mathcal M}(\Theta_n,W), \phi_j\vDash \phi_j$ for all $\phi_j\in W$.\\

$(2)$ ${\mathcal M}(\Theta_n,W)\vDash \displaystyle\bigvee_{i\in I}\phi_i\vee \displaystyle \bigvee_{j\in J}\phi_j$.\\

$(3)$ There is $i\in I$ such that $\phi_i\in W$ and ${\mathcal M}(\Theta_n,W), \phi_i\nvDash\displaystyle\bigvee_{j\in J}\phi_j$.\\

$(4)$~$r$ is invalid in ${\mathcal M}(\Theta_n,W)$.\\
\end{theorem}
\begin{proof}
Let $r$ is invalid in ${\mathcal N}(N,R)$ then ${\mathcal N}\vDash \displaystyle\bigvee_{i\in I}\phi_i\vee\displaystyle\bigvee_{j\in J}\phi_j$ and ${\mathcal N}\nvDash\displaystyle\bigvee_{j\in J}\phi_j$. Let $W=\{\phi_i\in \Theta_n\mid \exists w\in N~~{\mathcal N},w\models \phi_i\}$. By Theorem 3.3, we get $(1)$ and $(2)$.\\

$(3)$ Since ${\mathcal N}\nvDash\displaystyle\bigvee_{j\in J}\phi_j$, then $W\nsubseteq \{\phi_i\in\Theta_n\mid i\in J\}$ and since $W\subseteq \{\phi_i\in\Theta_n\mid i\in I\cup J\}$ then there is an $i\in I-J$ such that $\phi_i\in W$. By $\vDash\phi_i\rightarrow\neg\displaystyle\bigvee_{j\in J}\phi_j$ and ${\mathcal M}(\Theta_n,W), \phi_i\vDash \phi_i$ we get ${\mathcal M}(\Theta_n,W), \phi_i\vDash\neg\displaystyle\bigvee_{j\in J}\phi_j$, then ${\mathcal M}(\Theta_n,W), \phi_i\nvDash\displaystyle\bigvee_{j\in J}\phi_j$.\\

$(4)$ By $(2)$ and $(3)$.\\
\end{proof}

\begin{theorem}\label{208}
A rule $r=\cfrac{\displaystyle\bigvee_{i\in I}\phi_i}{p_1}$ is inadmissible for $S4$ iff there is a set $W\subseteq \{\phi_i\mid i\in I\}$ such that\\
$(1)$ There is $\phi_j\in W$ such that ${\mathcal M}(\Theta_n,W), \phi_j\nvDash p_1$.\\

$(2)$~${\mathcal M}(\Theta_n,W), \phi_j\vDash \phi_j$ for all $\phi_j\in W$.\\

$(3)$ For each subset $D$ of $W$ there exists $\phi_j\in W$ such that\\
\begin{center}
$\theta_{\diamond}(\phi_j)=\theta(\phi_j)\cup\displaystyle\bigcup_{\phi\in{\mathcal D}}\theta_{\diamond}(\phi)$.
\end{center}
Note that in (3),${\mathcal D}$ can be empty.\\
\end{theorem}

\begin{proof}
$(\Rightarrow)$ Let $r$ be inadmissible for $S4$, then by Theorem 3.1 there is an $n$-characterizing $S4$ model $Ch_{S4}(n)$ in which $r$ is invalid and by its way of construction the above condition $(3)$ holds for it. Let $W =\{\phi_i\in\Theta_n|\exists w\in Ch_{S4}(n)~~~~Ch_{S4}(n),w\models \phi_i\}$. Then the conditions $(1)-(3)$ hold for ${\mathcal M}(\Theta_n,W)$ by Theorem 3.3.\\
$(\Leftarrow)$ By $(1)$ and $(2)$, the rule $r$ is invalid in ${\mathcal M}(\Theta_n,W)$, and by the proof of Theorem 3.4.10 \cite{Sergey1} there is an extension of ${\mathcal M}(\Theta_n,W)$ to an $n$-characterizing $S4$ model $Ch_{S4}(n)$ with effectively contractible definable valuation $S$ of the variables $p_1,...,p_n$ which coincides with the original valuation of ${\mathcal M}(\Theta_n,W)$ and $Ch_{S4}(n)\models \displaystyle\bigvee_{i\in I}\phi_i $. Since ${\mathcal M}(\Theta_n,W)\nvDash p_1 $ and $Ch_{S4}(n)$ is an extension of ${\mathcal M}(\Theta_n,W)$ then $Ch_{S4}(n)\nvDash p_1$. Then $Ch_{S4}(n)\nvDash r$, therefore $r$ is inadmissible by Theorem 3.1
\end{proof}

\begin{theorem}\label{209}
A rule $r=\frac{\displaystyle\bigvee_{i\in I}\phi_i\vee \displaystyle\bigvee_{j\in J}\phi_j }{\displaystyle\bigvee_{j\in J}\phi_j}$ is inadmissible for $S4$ iff there is a set $W\subseteq \{\phi_i\in \Theta_n\mid i\in I\cup J\}$ such that\\
$(1)$~$\phi_i\in W$ for some $i\in I$.\\

$(2)$~${\mathcal M}(\Theta_n,W), \phi_j\vDash \phi_j$ for all $\phi_j\in W$.\\

$(3)$ For each subset $\mathcal(D)$ of ${\mathcal M}$ there exists $\phi_j\in W$ such that\\
\begin{center}
$\theta_{\diamond}(\phi_j)=\theta(\phi_j)\cup\displaystyle\bigcup_{\phi\in{\mathcal D}}\theta_{\diamond}(\phi)$.
\end{center}
\end{theorem}

\begin{proof}
$(\Rightarrow)$ Similar to the previous theorem.\\
 $(\Leftarrow)$ Similar to the previous theorem except $Ch_{S4}(n)\models \displaystyle\bigvee_{i\in I}\phi_i\vee \displaystyle\bigvee_{j\in J}\phi_j $ and since ${\mathcal M}(\Theta_n,W)\nvDash \displaystyle\bigvee_{j\in J}\phi_j $ and $Ch_{S4}(n)$ is an extension of ${\mathcal M}(\Theta_n,W)$ then $Ch_{S4}(n)\nvDash \displaystyle\bigvee_{j\in J}\phi_j$.\\
\end{proof}

\begin{theorem}\label{210}
Rules $r_1=\frac{\displaystyle\bigvee_{i\in I_1}\phi_i\vee \displaystyle\bigvee_{j\in J_1}\phi_j }{\displaystyle\bigvee_{j\in J_1}\phi_j},...,r_m=\frac{\displaystyle\bigvee_{i\in I_m}\phi_i\vee \displaystyle\bigvee_{j\in J_m}\phi_j }{\displaystyle\bigvee_{j\in J_m}\phi_j}$ are invalid in the same $S4$ model iff there is a set $W\subseteq \{\phi_i\in \Theta_n\mid i\in I_k\cup J_k\}$ for $k=1,...,m$ such that\\

$(1)$ $\phi_i\in W$ for some $i\in I_k$ for $k=1,...,m$.\\

$(2)$ ${\mathcal M}(\Theta_n,W), \phi_j\models \phi_j$ for all $\phi_j\in W$.\\
then $r_1,...,r_m$ are invalid ${\mathcal M}(\Theta_n,W)$.
\end{theorem}

\begin{theorem}\label{211}
Rules $r_1=\frac{\displaystyle\bigvee_{i\in I_1}\phi_i\vee \displaystyle\bigvee_{j\in J_1}\phi_j }{\displaystyle\bigvee_{j\in J_1}\phi_j},...,r_m=\frac{\displaystyle\bigvee_{i\in I_m}\phi_i\vee \displaystyle\bigvee_{j\in J_m}\phi_j }{\displaystyle\bigvee_{j\in J_m}\phi_j}$ are inadmissible for $S4$ with the same substitution $\sigma$ iff there is a set $W\subseteq \{\phi_i\in \Theta_n\mid i\in I_k\cup J_k\}$ for $k=1,...,m$ such that\\

$(1)$ $\phi_i\in W$ for some $i\in I_k$ for $k=1,...,m$.\\

$(2)$~${\mathcal M}(\Theta_n,W), \phi_j\models \phi_j$ for all $\phi_j\in W$.\\

$(3)$ For each subset $\mathcal(D)$ of ${\mathcal M}$ there exists $\phi_j\in W$ such that\\
\begin{center}
$\theta_{\diamond}(\phi_j)=\theta(\phi_j)\cup\displaystyle\bigcup_{\phi\in{\mathcal D}}\theta_{\diamond}(\phi)$.
\end{center}
\end{theorem}

\begin{proof}
$(\Rightarrow)$ Let $r_1,...,r_m$ are inadmissible with the same $\sigma$. $r_1,...,r_m$ are invalid in an $n$-characterizing $S4$ model $Ch_{S4}(n)$ by Theorem $3.1$. Let $W=\{\phi_i\in\Theta_n\mid \exists x\in Ch_{S4}(n)~~Ch_{S4}(n),x\vDash \phi_i\}$, then $W\subseteq \{\phi_i\in \Theta_n\mid i\in I_k\cup J_k\}$ for $k=1,...,m$. Similar to the previous theorems, the conditions $(1)-(3)$ hold for ${\mathcal M}(\Theta_n,W)$.\\
$(\Leftarrow)$ By $(1)$ and $(2)$, the rule $r'_k=\frac{\displaystyle\bigvee_{i\in W}\phi_i}{\displaystyle\bigvee_{j\in J_k}\phi_j}$ is invalid in ${\mathcal M}(\Theta_n,W)$, and by the proof of Theorem 3.4.10 \cite{Sergey1} there is an extension of ${\mathcal M}(\Theta_n,W)$ to an $n$-characterizing $S4$ model $Ch_{S4}(n)$ with effectively contractible definable valuation $S$ of the variables $p_1,...,p_n$ which coincides with the original valuation of ${\mathcal M}(\Theta_n,W)$ and $Ch_{S4}(n)\models \displaystyle\bigvee_{i\in W}\phi_i $ then $Ch_{S4}(n)\models \displaystyle\bigvee_{i\in I_k}\phi_i\vee\displaystyle\bigvee_{j\in J_k}\phi_j $. Since ${\mathcal M}(\Theta_n,W)\nvDash \displaystyle\bigvee_{j\in J_k}\phi_j $ and $Ch_{S4}(n)$ is an extension of ${\mathcal M}(\Theta_n,W)$ then $Ch_{S4}(n)\nvDash \displaystyle\bigvee_{j\in J_k}\phi_j$. Then $Ch_{S4}(n)\nvDash r_1,...,Ch_{S4}(n)\nvDash r_m$, therefore $r_1,...,r_m$ are inadmissible by the same $\sigma$ by Theorem $3.1$.
\end{proof}

\begin{theorem}\label{212}
Rules $r_1=\frac{\displaystyle\bigvee_{i\in I_1}\phi_i\vee \displaystyle\bigvee_{j\in J_1}\phi_j }{\displaystyle\bigvee_{j\in J_1}\phi_j},...,r_m=\frac{\displaystyle\bigvee_{i\in I_m}\phi_i\vee \displaystyle\bigvee_{j\in J_m}\phi_j }{\displaystyle\bigvee_{j\in J_m}\phi_j}$ are inadmissible for $S4$ with the same substitution $\sigma$ iff there is a set $W\subseteq \{\phi_i\in \Theta_n\mid i\in I_k\cup J_k\}$ for $k=1,...,m$ such that $r'_1=\frac{\displaystyle\bigvee_{\phi_i\in W}\phi_i}{\displaystyle\bigvee_{j\in J_1}\phi_j},...,r'_m=\frac{\displaystyle\bigvee_{\phi_i\in W}\phi_i}{\displaystyle\bigvee_{j\in J_m}\phi_j}$ are inadmissible for $S4$ with the same substitution $\sigma$.\\
\end{theorem}

\begin{proof}
$(\Leftarrow)$ is trivial.\\

$(\Rightarrow)$ Let $\phi_j\notin \{\phi_i\in \Theta_n\mid i\in I_l\cup J_l\} $ then for $i\in I_l\cup J_l$ we get $\vdash\phi_i\rightarrow \neg \phi_j $ then $\vdash\displaystyle\bigvee_{i\in I_l\cup J_l}\sigma(\phi_i)\rightarrow \neg \sigma(\phi_j) $ and since $\vdash\displaystyle\bigvee_{i\in I_l\cup J_l}\sigma(\phi_i)$ so $\vdash \neg \sigma(\phi_j)$. For $k=1,...,m$, since $\vdash\displaystyle\bigvee_{i\in I_k\cup J_k}\sigma(\phi_i)$ then $\vdash\displaystyle\bigvee_{i\in I_k\cup J_k-\{j\}}\sigma(\phi_i)$ therefore the rule $\frac{\displaystyle\bigvee_{i\in I_k\cup J_k-\{j\}}\sigma(\phi_i)}{\displaystyle\bigvee_{j\in J_k}\phi_j}$ is inadmissible for $S4$ with the  substitution $\sigma$.
\end{proof}

In the previous theorems, in the sets $W$ there must be appropriate formula $\phi_i$ to support conditions such as ${\mathcal M}(\Theta_n,W), \phi_j\vDash \phi_j$ for all $\phi_j\in W$. In the end of this section we study the construction and properties of these sets.\\

\begin{definition}\label{213}
$ Supp_1\subseteq {\mathcal P}(\Theta_n) $ is the smallest set where

$(1)$ If $C_1,...,C_m$ are clusters of ${\mathcal M}_n$, that is, $\forall x,y \in C_i~~~ \Theta_{\diamond}(x)=\Theta_{\diamond}(y)=\Theta_{\diamond}(C_i)$ for each $i=1,...,m$, and $X_i\subseteq C_i$ where $\displaystyle\bigcup_{x\in X_i}\Theta(x)=\Theta_{\diamond}(C_i)$ for each $i=1,...,m$, then $\displaystyle\bigcup_{i=1}^{m} X_i\in Supp_1$.\\

$(2)$ If $Y_i\subseteq W\in Supp_1$ and $X_i\subseteq C_i$ where $C_i$ is a cluster, $\forall y\in Y_i~~\Theta_{\diamond}(y)\subseteq \Theta_{\diamond}(C_i) $ and $\Theta_{\diamond}(C_i)=\displaystyle\bigcup_{y\in Y_i}\Theta(y)\cup\displaystyle \bigcup_{x\in X_i}\Theta(x)$ for $i=1,...,n$, then $W\cup\displaystyle\bigcup_{i=1}^{n} X_i\in Supp_1$.\\

$Supp_2\subseteq Supp_1$ is defined as $W\in Supp_2$ if in addition

$(3)$ For each $\mathcal{D}\subseteq W$ there is $z\in W$ such that $\Theta_{\diamond}(z)=\Theta(z)\cup \displaystyle\bigcup_{x\in \mathcal{D}}\Theta_{\diamond}(x)$.
\end{definition}

\begin{Ex}
 Let $W_1=\{55\}$, $W_2=\{46\}$, $W_3=\{55,46\}$, $W_4=\{55\}\cup\{13,14\}$

$W_5=\{30,31\}\cup\{21,23\}\cup\{13,14\}$, $W_6=\{21,55\}$, $W_7=\{21,30,55\}$,

$W_8=\{8,21,30,55\}$.\\
These sets are in $Supp_1$ and $W_1,W_2,W_6,W_8\in Supp_2$ but $W_3,W_4,W_5,W_7\notin Supp_2 $.\\
\end{Ex}

\begin{Rem}
If $n\leq 3$, the condition $(3)$ can be replaced with\\
$(4)$ $\exists x\in W$ where $\Theta(x)=\Theta_{\diamond}(x)$.\\
$(5)$ If $x,y\in W$ then $zRx$ and $zRy$ for some $ z\in W $ with $\Theta_{\diamond}(z)\subseteq \Theta(z)\cup \Theta_{\diamond}(x)\cup \Theta_{\diamond}(y)$ .\\

But for $n\geq 4$, let $W=\{x_1,x_2,x_3,y_1,y_2,y_3,z\}$ with $\Theta(x_i)=\Theta_{\diamond}(x_i)=\{p_i\},~i=1,2,3,~\Theta_{\diamond}(y_1)=\{p_1,p_2,p_4\},~\Theta_{\diamond}(y_2)=\{p_1,p_3,p_4\}, ~\Theta_{\diamond}(y_3)=\{p_2,p_3,p_4\},~\Theta_{\diamond}(z)=\{p_1,p_2,p_3,p_4\}$ and $\Theta(y_1)=\Theta(y_2)=\Theta(y_3)=\{p_4\}, \Theta(z)=\emptyset.$ Then $W\in Supp_1$ and the conditions $(4), (5)$ are true for $W$ but $(3)$ is false for $\mathcal{D}=\{x_1,x_2,x_3\}$.
\end{Rem}

\begin{theorem}
$(1)$ If $x\in W\in Supp_1$ and $W'=\{w\in W \mid xRw \}$ then $W'\in Supp_1$.\\

$(2)$ If $x\in W\in Supp_1$ and $W'=\{w\in W \mid xRw \}$ and $\Theta_\diamond(x)=\Theta(x)$ then $W'\in Supp_2$.\\

$(3)$ If $x\in W\in Supp_1$ and $\Theta_{\diamond}(x)=\Theta_{\diamond}(y)$, then $W\cup\{y\}\in Supp_1$.\\

%$(4)$ Let $W'=\{x\in W\mid \exists y\in W~~(yRx\wedge \neg xRy)\}$ and $W_x=\{y\in W\mid \Theta_\diamond(x)=\Theta_\diamond(y)\}\cup W' $ for each $x\in W-W'=\{x_1,...,x_m\}$, and $\overline{W}_0=W'$ and $\overline{W}_i=\overline{W}_{i-1}\cup W_{x_i}$ for $i=1,...,m$. Then\\

%$(i)$~$W=\displaystyle \bigcup_{x\in W-W'}W_x=\displaystyle \bigcup_{i=1}^{m}\overline{W}_i=\overline{W}_m$.\\

%$(ii)$ If $W'=\emptyset$ then $W_x=\{y\in W\mid \Theta_\diamond(x)=\Theta_\diamond(y)\}$ and $W\in Supp_1~~iff~~\overline{W}_i\in Supp_1$ for $i=1,...,m~~iff~~W_x\in Supp_1$ for each $x\in W-W'$ by definition 2(1). \\

%$(iii)$ If $W'\neq \emptyset$, then $W\in Supp_1~~iff~~\overline{W}_i\in Supp_1$ for $i=0,...,m~~iff~~W_x\in Supp_1$ for each $x\in W-W'$ by definition 2(2).\\

$(4)$ If $W_1,W_2\in Supp_1$ then $W_1\cup W_2\in Supp_1$.\\

%$(5)$ If $W_1\in Supp_1, W_2\in Supp_2$ and $x\in W_1\cap W_2$ with $\forall y\in W_1\cup W_2~~~xRy$ then $W_1\cup W_2\in Supp_2$.\\

$(5)$~$W\in Supp_1$ iff ${\mathcal M}(\Theta_n,W), \phi_j\vDash \phi_j$ for all $\phi_j\in W$.\\

$(6)$ $W\in Supp_2$ iff

$(i)$~${\mathcal M}(\Theta_n,W), \phi_j\vDash \phi_j$ for all $\phi_j\in W$ and

$(ii)$ For each subset $\mathcal(D)$ of $W$ there exists $\phi_j\in W$ such that
\begin{center}
$\theta_{\diamond}(\phi_j)=\theta(\phi_j)\cup\displaystyle\bigcup_{\phi\in{\mathcal D}}\theta_{\diamond}(\phi)$.
\end{center}

$(7)$~$r=\frac{I,J}{J}$ is invalid iff there is $W\in Supp_1$ such that $W\cap I\neq \emptyset$ and $W\subseteq I\cup J$.\\

$(8)$~$r=\frac{I,J}{J}$ is inadmissible iff there is $W\in Supp_2$ such that $W\cap I\neq \emptyset$ and $W\subseteq I\cup J$.\\

$(9)$~If $r=\frac{I,J}{J}$ is valid then $r=\frac{I,J}{J}$ is admissible.\\

$(10)$~If $r=\frac{I,J}{J}$ is invalid then $r'=\frac{x,I,J}{x,J}$ is inadmissible for some $x$ with $\theta(x)=\theta_{\diamond}(x)$. \\

$(11)$ If $r_1=\frac{I_1,J_1}{J_1},...,r_k=\frac{I_k,J_k}{J_k}$ are invalid in the same $S4$-model iff there is $W\in Supp_1$ such that $W\cap I_i\neq \emptyset$ and $W\subseteq I_i\cup J_i$ for $i=1,...,k$.\\

$(12)$ If $r_1=\frac{I_1,J_1}{J_1},...,r_k=\frac{I_k,J_k}{J_k}$ are inadmissible with the same $\sigma$ iff there is $W\in Supp_2$ such that $W\cap I_i\neq \emptyset$ and $W\subseteq I_i\cup J_i$ for $i=1,...,k$.\\

$(13)$~If $r_1=\frac{I_1,J_1}{J_1},...,r_k=\frac{I_k,J_k}{J_k}$ are invalid in the same $S4$-model then $r'_1=\frac{x,I_1,J_1}{x,J_1},...,r'=\frac{x,I_1,J_1}{x,J_1}$ are inadmissible with the same $\sigma$ for some $x$ with $\theta(x)=\theta_{\diamond}(x)$. \\

\end{theorem}

\begin{proof}

$(1)$~$(i)$ Let $C_1,...,C_m$ be clusters of ${\mathcal M}_n$ and $X_i\subseteq C_i$ where $\displaystyle\bigcup_{x\in X_i}\Theta(x)=\Theta_{\diamond}(C_i)$ for each $i=1,...,m$, and $x\in W=\displaystyle\bigcup_{i=1}^{m} X_i$. Let $x\in X_i\subseteq C_i$ Then $W'=\displaystyle\bigcup_{\forall y \in X_j xRy } X_j\in Supp_1$. \\

$(ii)$ Let $Y_i\subseteq W_1\in Supp_1$ and $X_i\subseteq C_i$ where $C_i$ is a cluster, $\forall y\in Y_i~~\Theta_{\diamond}(y)\subseteq \Theta_{\diamond}(C_i) $ and $\Theta_{\diamond}(C_i)=\displaystyle\bigcup_{y\in Y_i}\Theta(y)\cup\displaystyle \bigcup_{x\in X_i}\Theta(x)$, for $i=1,...,n$, and $x\in W=W_1\cup\displaystyle\bigcup_{i=1}^{n} X_i$. Then $W'=W'_1\cup\displaystyle\bigcup_{\forall y \in X_j xRy } X_j$ and by induction hypothesis $W'_1\in Supp_1$, and if $x$ is in some $C_i$, then $\forall y\in Y_i~~xRy$ and $Y_i\subseteq W'_1$, so $W'\in Supp_1$  by Definition $3.3(2)$.\\

$(2)$ By $(1), W'\in Supp_1$ and if $\mathcal{D}\subseteq W'$ and $y\in\mathcal{D}$ then $\Theta_{\diamond}(y)\subseteq\Theta_{\diamond}(x)$ then $\displaystyle \bigcup_{y\in \mathcal{D}}\Theta_{\diamond}(y)\subseteq \Theta_{\diamond}(x)=\Theta(x)$, therefore $\Theta_{\diamond}(x)=\Theta(x)\cup \displaystyle \bigcup_{y\in \mathcal{D}}\Theta_{\diamond}(y)$. \\

$(3)$~$(i)$ Let $C_1,...,C_m$ be clusters of ${\mathcal M}_n$ and $X_i\subseteq C_i$ where $\displaystyle\bigcup_{x\in X_i}\Theta(x)=\Theta_{\diamond}(C_i)$ for each $i=1,...,m$, and $x\in W=\displaystyle\bigcup_{i=1}^{m} X_i$. Let $x\in X_j\subseteq C_j$ and $\Theta_{\diamond}(x)=\Theta_{\diamond}(y)$ Then $X_j\cup \{y\}\subseteq C_j$ and $\displaystyle\bigcup_{x\in X_j\cup \{y\}}\Theta(x)=\Theta_{\diamond}(C_j)$. Then $W\cup\{y\}=\displaystyle\bigcup_{j\neq i=1}^{m} X_i\cup (X_j\cup \{y\})\in Supp_1$. \\

$(ii)$  Let $Y_i\subseteq W_1\in Supp_1$ and $X_i\subseteq C_i$ where $C_i$ is a cluster, $\forall y\in Y_i~~\Theta_{\diamond}(y)\subseteq \Theta_{\diamond}(C_i) $ and $\Theta_{\diamond}(C_i)=\displaystyle\bigcup_{y\in Y_i}\Theta(y)\cup\displaystyle \bigcup_{x\in X_i}\Theta(x)$, for $i=1,...,n$, and $x\in W=W_1\cup\displaystyle\bigcup_{i=1}^{n} X_i$. If $x\in W_1$, then by induction hypothesis $W_1\cup\{y\}\in Supp_1$, so $ W\cup \{y\} =(W_1\cup\{y\})\cup \displaystyle\bigcup_{i=1}^{n} X_i$  by Definition $3.3(2)$. If $x\in X_j$, then $X_j\cup\{y\}\subseteq C_j$ and $\displaystyle \bigcup_{w\in X_j}\Theta(w)\subseteq \displaystyle \bigcup_{w\in X_j\cup\{y\}}\Theta(w)\subseteq \Theta_{\diamond}(C_j) $ thus $\Theta_{\diamond}(C_j)=\displaystyle\bigcup_{w\in Y_j}\Theta(w)\cup\displaystyle \bigcup_{w\in X_j\cup\{y\}}\Theta(w)$. Therefore $W\cup\{y\} =W_1\cup (X_j\cup\{y\})\cup\displaystyle\bigcup_{j\neq i=1}^{n} X_i\in Supp_1$ by Definition $3.3(2)$. \\

%$(4)$ $(i)$ and $(ii)$ are trivial and $(iii)$ by induction on the construction of $W$ in Definition $2$. If $W\in Supp_1$ by Definition $2(1)$, then $W'=\emptyset$. Let $W=W_1\cup Y$,~~~$X\subseteq W_1\in Supp_1$,~~~$Y\subseteq C$ and $C$ be a cluster with $\Theta_{\diamond}(C)=\displaystyle\bigcup_{x\in X}\Theta(x)\cup\displaystyle \bigcup_{y\in Y}\Theta(y)$, and by $(3)$ we can assume $C\cap W_1=\emptyset$. Then $W'\subseteq W_1\subseteq W$ and we can assume $W=\overline{W}_i$ for $k\leq i\leq m$ and $W_1=\overline{W}_i$ for $l\leq i\leq k$ and $(\overline{W}_1)_i=\overline{W}_i$ for $i\leq l$. Then by induction hypothesis $\overline{W}_i\in Supp_1$ for $i=1,...,m$.\\

%$(5)$ By induction on the number of elements of $W_1\cup W_2$. If $W_1, W_2\in Supp_1$ by Definition 2(1), so is $W_1\cup W_2\in Supp_1$ by Definition 2(1). Otherwise, if $W_2=W_3\cup\displaystyle\bigcup_{i=1}^{n} X_i\in Supp_1 $ with $Y_i\subseteq W_3\in Supp_1$ and $X_i\subseteq C_i$ where $C_i$ is a cluster, $\forall y\in Y_i~~\Theta_{\diamond}(y)\subseteq \Theta_{\diamond}(C_i) $ and $\Theta_{\diamond}(C_i)=\displaystyle\bigcup_{y\in Y_i}\Theta(y)\cup\displaystyle \bigcup_{x\in X_i}\Theta(x)$ for $i=1,...,n$, then by induction hypothesis $Y_i\subseteq W_1\cup W_3\in Supp_1$ so $W_1\cup W_2=W_1\cup W_3\cup\displaystyle\bigcup_{i=1}^{n} X_i\in Supp_1$ by Definition 2(2).

$(5)$~$(\Leftarrow)$ By induction on $W$. Let $W=\displaystyle\bigcup_{i=1}^{n} X_i$ where $\emptyset\neq X_i\subseteq C_i$ and $C_i$ is a cluster for $i=1,...,n$. $\displaystyle\bigcup_{x\in X_i} \Theta(x)\subseteq \Theta_{\diamond}(C_i)$ is trivial. If $p_k\in\Theta_{\diamond}(C_i)$ and $\varphi_j\in X_i\subseteq C_i$, since $\varphi_j\vDash\varphi_j$ and $p_k\in\Theta_{\diamond}(\varphi_j)=\Theta_{\diamond}(C_i) $ then $\varphi_j\vDash \diamond p_k$. Let $\varphi_j R \varphi_{j'}$ and $\varphi_{j'}\vDash p_k$ then $p_k\in \Theta(\varphi_{j'})$. If $W'=\emptyset$ then $\varphi_{j'}\in X_i$ thus $p_k\in\displaystyle\bigcup_{x\in X_i} \Theta(x)$, so $\displaystyle\bigcup_{x\in X_i} \Theta(x)=\Theta_{\diamond}(C_i)$ and therefore $W\in Supp_1$ by Definition $2(1)$. If $W'\neq\emptyset$ and $X_i\subseteq W-W'$ and $Y_i=\{\varphi_{j'}\in W'\mid \varphi_j\in X_i \}$, then $\varphi_{j'}\in X_i\cup Y_i$ and $p_k\in\displaystyle\bigcup_{x\in X_i} \Theta(x)\cup\displaystyle\bigcup_{y\in Y_i} \Theta(y)$, so $\displaystyle\bigcup_{x\in X_i} \Theta(x)\cup\displaystyle\bigcup_{y\in Y_i} \Theta(y)=\Theta_{\diamond}(C_i)$. Since the hypothesis ${\mathcal M}(\Theta_n,W'),\varphi_j\vDash \varphi_j$ is true for all $\varphi_j\in W'$, by induction hypothesis $W'\in Supp_1$ and therefore $W=W'\cup\displaystyle\bigcup_{X_i\subseteq W-W'}X_i\in Supp_1$ by Definition $3.3(2)$

$(\Rightarrow)$is trivial.

$(6)$ It is trivial by $(5)$ and Definition $3.3(3).$

$(7)$ By the part $(5)$ and Theorem $3.4$.\\

$(8)$ By the part $(6)$ and Theorem $3.6$. \\

$(9)$ If $r=\frac{I,J}{J}$ is inadmissible then by the part $(8)$, there is $W\in Supp_2$ such that $W\cap I\neq \emptyset$ and $W\subseteq I\cup J$. Since $Supp_2\subseteq Supp_1$ then $W\in Supp_1$ so $r$ is invalid by the part $(7)$.\\

$(10)$~If $r=\frac{I,J}{J}$ is invalid then by the part $(7)$, there is $W\in Supp_1$ such that $W\cap I\neq \emptyset$ and $W\subseteq I\cup J$. Let $z\in W\cap I$ and $W'=\{w\in W\mid zRw\}$ and $\theta(x)=\theta_{\diamond}(x)=\theta_{\diamond}(z)$.  By the parts $(2)$ and $(3)$, $W'\cup \{x\}\in Supp_2$, $(W'\cup \{x\})\cap I\neq \emptyset$ and $(W'\cup \{x\})\subseteq (I\cup \{x\})\cup J$. Then by the part $(8)$ $r'=\frac{x,I,J}{x,J}$ is inadmissible.\\

$(11)-(13)$ similar to $(7)-(10)$ by Theorems $3.8$ and $3.9$.
\end{proof}

\section{substitutions for  rules in reduced normal forms}
Suppose $S:{\mathcal P}(\Theta_n)\rightarrow {\mathcal P}(\Sigma)$ with $S(A)=\{\sigma~|~\nvdash_{\sigma}r~~~for~ each~r~in~A\}$. The following theorem shows properties of $S$ and the way of composing $S(A)$ to its components in the form of $S(\frac{W}{W-\{i_1\}},...,\frac{W}{W-\{i_n\}})$ where $W=\{i_1,...,i_n\}$ according to the parts $(7)$ and $(8)$ of the theorem. Some applications of the theorem is given by figures $(4)$, $(5)$ and $(6)$ and example $(4.1)$. The compositions are done using the actions a $+J$ and $-J$ in the parts $(12)$ and $(13)$, decreasing premisses and adding demands $(9)$ and $(11)$, and simplifications $(4)$, $(5)$ and $(6)$. The method introduced in the theorem and the following examples justify the approach of the paper in considering substitutions rejecting a set of rules instead of a single rule as usual and generalizing the results in the previous sections.

\begin{theorem}\label{315}
 The above function $S$ satisfies the following:\\

$(1)$~If $A\subseteq B$ then $S(B)\subseteq S(A)$.\\

$(2)$~$S(A)\cap S(B)=S(A\cup B)$.\\

$(3)$~$S(A)\cup S(B)\subseteq S(A\cap B)$.\\

$(4)$~ $S(\{\frac{W_1}{J_1},\frac{W_2}{J_2}\}\cup A)=S(\{\frac{W_2}{J_2}\}\cup A)$ if $J_1\subseteq J_2, W_2\subseteq W_1.$\\

$(5)$ If $J_i\subseteq J'_i, W_i\subseteq W'_i$ and $\frac{W'_i}{W_i}$ and $\frac{J'_i}{J_i}$ are admissible for $i=1,...,m$ then $S(\frac{W_1}{J_1},...,\frac{W_m}{J_m})=S(\frac{W'_1}{J'_1},...,\frac{W'_m}{J'_m})$ (special cases of $(10)$ and $(11)$).\\

$(6)$ If $W=\displaystyle\bigcap_{i=1}^{m}W_i $ then $S(\{\frac{W_1}{J_1},...,\frac{W_m}{J_m}\})=S(\{\frac{W}{J_1\cap W},...,\frac{W}{J_m\cap W}\})$ then $S(\{\frac{W_1}{J_1},...,\frac{W_m}{J_m}\})=\varnothing$ if $J_1\subseteq W$ or ... or $J_m\subseteq W$\\

$(7)$ If $W=\{i_1,...,i_n\}$ and $J_1,...,J_m\subseteq W$ then $S(\frac{W}{W-\{i_1\}},...,\frac{W}{W-\{i_n\}})\subseteq S(\frac{W}{J_1},...,\frac{W}{J_m})$\\

$(8)$ If $W_1=\{i_1,...,i_n\}$ and $W_2=\{j_1,...,j_m\}$ and $W_1\neq W_2$ then\\

$S(\frac{W_1}{W_1-\{i_1\}},...,\frac{W_1}{W_1-\{i_n\}})\cap S(\frac{W_2}{W_2-\{j_1\}},...,\frac{W_2}{W_2-\{j_m\}})=\emptyset$\\

$(9)$~If $J_1,...,J_n\varsubsetneq W_1\varsubsetneq W_2\subseteq \Theta(r)$ then \\

$(i)$ $S(\{\frac{W_2}{J_1},...,\frac{W_2}{J_m}\})=S(\{\frac{W_1}{J_1},...,\frac{W_1}{J_m}\})\bigcup S(\{\frac{W_2}{W_1}\})$.\\

$(ii)$ $S(\{\frac{W_1}{J_1},...,\frac{W_1}{J_m}\})\cap S(\{\frac{W_2}{W_1}\})=\emptyset$.\\

$(iii)$ $S(\{\frac{W_1}{J_1},...,\frac{W_1}{J_m}\})\subseteq S(\{\frac{W_2}{J_1},...,\frac{W_2}{J_m}\})$.\\

$(10)$ If $J_i\varsubsetneq W_i\subseteq W'_i\subseteq \Theta(r)$ for $i=1,...,m$ then \\

$(i)$ $S(\{\frac{W'_1}{J_1},...,\frac{W'_m}{J_m}\})=S(\{\frac{W_1}{J_1},...,\frac{W_m}{J_m}\})\bigotimes S(\{\frac{W'_1}{W_1},...,\frac{W'_m}{W_m}\}):=\displaystyle\bigcup_{\alpha\in 2^{m}}S_{\alpha}$ where\\

$S_{\alpha}=S(\{r_1,...,r_m\})$ and
\[ r_i = \begin{cases} \frac{W_i}{J_i} & \quad \text{if } \alpha(i-1)=0\\ \frac{W'_i}{W_i} & \quad \text{if } \alpha(i-1)=1\\ \end{cases} \]
for $i=1,...,m$.\\

$(ii)$~$S(\alpha)\cap S(\beta)=\emptyset$ if $\alpha\neq\beta.$\\

$(iii)$~ $S(\{\frac{W_1}{J_1},...,\frac{W_m}{J_m}\})\cap S(\{\frac{W'_1}{W_1},...,\frac{W'_m}{W_m}\})=\emptyset$.\\

 $(iv)$~$S(\{\frac{W_1}{J_1},...,\frac{W_m}{J_m}\})\subseteq S(\{\frac{W'_1}{J_1},...,\frac{W'_m}{J_m}\})$.\\

$(11)$~If $J_i\subseteq J'_i\subseteq W_i\subseteq \Theta_n(r)$ for $i=1,...,m$ then \\

$(i)$~$S(\{\frac{W_1}{J_1},...,\frac{W_m}{J_m}\})=S(\{\frac{W_1}{J'_1},...,\frac{W_m}{J'_m}\})\bigotimes S(\{\frac{J'_1}{J_1},...,\frac{J'_m}{J_m}\}):=\displaystyle\bigcup_{\alpha\in 2^{m}}S_{\alpha}$ where $S_{\alpha}=S(\{r_1,...,r_m\})$ and
\[ r_i = \begin{cases} \frac{W_i}{J'_i} & \quad \text{if } \alpha(i-1)=0\\ \frac{J'_i}{j_i} & \quad \text{if } \alpha(i-1)=1\\ \end{cases} \] for $i=1,...,m$.\\

$(ii)$ $S(\alpha)\cap S(\beta)=\emptyset$ if $\alpha\neq\beta.$\\

$(iii)$ $S(\{\frac{W_1}{J'_1},...,\frac{W_m}{J'_m}\})\cap S(\{\frac{J'_1}{J_1},...,\frac{J'_m}{J_m}\})$.\\

$(iv)$ $S(\{\frac{W_1}{J'_1},...,\frac{W_m}{J'_m}\})\subseteq S(\{\frac{W_1}{J_1},...,\frac{W_m}{J_m}\})$.\\

$(12)$ If $J_i\subsetneq W_i\subseteq \Theta(r)$  for $i=1,...,m$, and $\{\frac{W_1}{J_1},...,\frac{W_m}{J_m}\}-J=\{\frac{W_1-J}{J_1-J},...,\frac{W_m-J}{J_m-J}\}$ then\\

$(i)$ $S(\{\frac{W_1}{J_1},...,\frac{W_m}{J_m}\})=S(\{\frac{W_1}{J_1},...,\frac{W_m}{J_m}\}-J)\cup \displaystyle\bigcup_{i=1}^{m} (S(\{\frac{W_1}{J_1},...,\frac{W_m}{J_m}\})\cap S(\{\frac{W_i}{W_i-J}\}))$\\

~~~~~~~~~~~~~~~~~~~~~~~~~~~$=S(\{\frac{W_1}{J_1},...,\frac{W_m}{J_m}\}-J)\cup\displaystyle\bigcup_{i=1}^{m} S(\{\frac{W_1}{J_1},...,\frac{W_m}{J_m},\frac{W_i}{W_i-J}\}) $.\\

$(ii)$ $S(\{\frac{W_1}{J_1},...,\frac{W_m}{J_m}\}-J)\cap S(\{\frac{W_1}{J_1},...,\frac{W_m}{J_m},\frac{W_i}{W_i-J}\})= \emptyset$.\\

$(iii)$ $S(\{\frac{W_1}{J_1},...,\frac{W_m}{J_m}\}-J)\subseteq S(\{\frac{W_1}{J_1},...,\frac{W_m}{J_m}\})$.\\

$(13)$ If $J\subseteq W_i\subseteq \Theta(r)$  for $i=1,...,m$, and $\{\frac{W_1}{J_1},...,\frac{W_m}{J_m}\}+J=\{\frac{W_1}{J_1\cup J},...,\frac{W_m}{J_m\cup J}\}$ then\\

$(i)$ $S(\{\frac{W_1}{J_1},...,\frac{W_m}{J_m}\})=S(\{\frac{W_1}{J_1},...,\frac{W_m}{J_m}\}+J)\bigotimes S(\{\frac{J_1\cup J}{J_1},...,\frac{J_m\cup J}{J_m}\}):=\displaystyle\bigcup_{\alpha\in 2^{m}}S_{\alpha}$ where
$S_{\alpha}=S(\{r_1,...,r_m\})$ and
\[ r_i = \begin{cases} \frac{W_i}{J_i\cup J} & \quad \text{if } \alpha(i-1)=0\\ \frac{J_i\cup J}{J_i} & \quad \text{if } \alpha(i-1)=1\\ \end{cases} \] for $i=1,...,m$.\\

$(ii)$ $S(\alpha)\cap S(\beta)=\emptyset$ if $\alpha\neq\beta.$\\

$(iii)$ $S(\{\frac{W_1}{J_1},...,\frac{W_m}{J_m}\}+J)\cap S(\{\frac{J_1\cup J}{J_1},...,\frac{J_m\cup J}{J_m}\})=\emptyset$.\\

$(iv)$ $S(\{\frac{W_1}{J_1},...,\frac{W_m}{J_m}\}+J)\subseteq S(\{\frac{W_1}{J_1},...,\frac{W_m}{J_m}\})$.\\
\end{theorem}
\begin{theorem}
If $J_i\subseteq W_i\subseteq \Theta_n(r)$ for $i=1,...,m$, then $S(\{\frac{W_1}{J_1},...,\frac{W_m}{J_m}\})\neq \emptyset$ iff
there is $W \in Supp_2$ such that  $J_i\subsetneq W\subseteq W_i$, for $i=1,...,m$.
\end{theorem}
\begin{proof}
By Theorem $3.10(12)$.
\end{proof}
\begin{Rem}
In the items $(9)-(13)$ in the Theorem $4.1$, by Theorem $4.2$,  we can replace $\subsetneq$ with $\subseteq$  if for some $W\in Supp_2$ we suppose $W_1\varsubsetneq W\subseteq W_2$ in $(9)$, $W_i\subsetneq W\subseteq W'_i$ for $i=1,...,m$ in $(10)$, $J_i\subsetneq W\subseteq J'_i$ for $i=1,...,m$ in $(11)$, $J_i\subsetneq W\subseteq W_i$ for $i=1,...,m$ in $(12)$. $J_i\varsubsetneq W \subseteq J_i\cup J$ for $i=1,...,m$ in $(13)$.
\end{Rem}

\begin{Rem}Theorem $4.1$ suggests a method to decompose $S(A)$ to its components as in $(7)$ and $(8)$. If for some $\frac{W}{J}\in A$, $i\in W-J$ and $W-J$ has another element, we use the part $(13)$ for $J=\{i\}$ and denote this action by
 $\stackrel{+i}{=}$. If $i\in W\cap J$ for some $\frac{W}{J}\in A$ and the previous action can't be done, we use $(12)$ for $J=\{i\}$ and denote this action by $\stackrel{-i}{=}$.
After each of these actions we use $(4)$ and $(5)$ and $(6)$ to simplify or unify the premisses or demands and to remove unnecessary rules. We denote this action by $\stackrel{s}{=}$. When the premisses are the same the action $\stackrel{-i}{=}$ is simplified to
\begin{center}
$S(\{\frac{W}{J_1},...,\frac{W}{J_m}\})=S(\{\frac{W-\{i\}}{J_1-\{i\}},...,\frac{W-\{i\}}{J_m-\{i\}}\})\cup S(\{\frac{W}{J_1},...,\frac{W}{J_m},\frac{W}{W-\{i\}}\})$.
\end{center}

If $i\in J'$ for each $\frac{W}{J'}\in A$ and $i\in W-J$ then the action $\stackrel{+i}{=}$ is simplified to

\begin{center}
$S(\{\frac{W}{J}\}\cup A)=S(\{\frac{W}{J\cup\{i\}}\}\cup A)\cup S(\{\frac{J\cup\{i\}}{J}\}\cup A)$.
\end{center}

By Theorem $4.2$ in each step the empty or nonempty sets are determined.

Figure $4$ shows the tree of applications of actions $\pm i$ for an example from Figure $2$. Figure $5$ shows the simplified forms of the rules in figure $4$. Since the tree for $3$ variables is very big, Figure $6$ shows a branch of application $\pm i$ for an example from figure $3$.
\end{Rem}
\begin{figure}
  \centering
    \includegraphics[scale=.6]{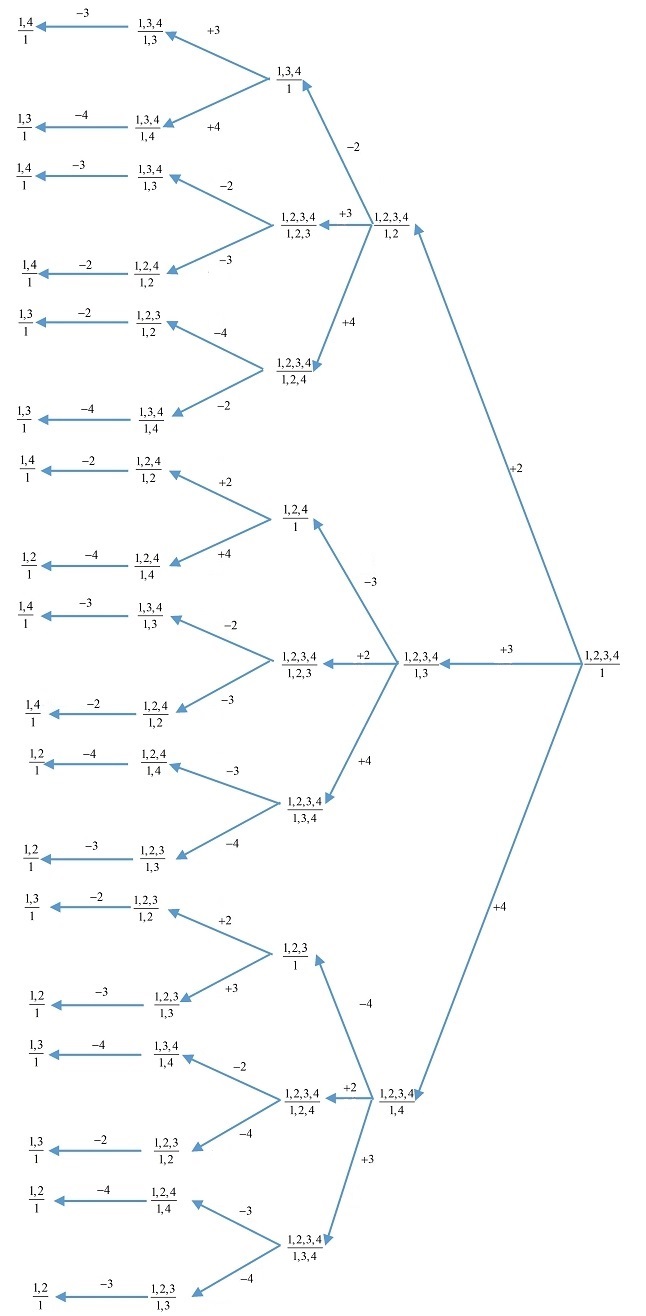}
  \caption{\label{fig:sad}}
\end{figure}
\begin{figure}
  \centering
    \includegraphics[scale=.4]{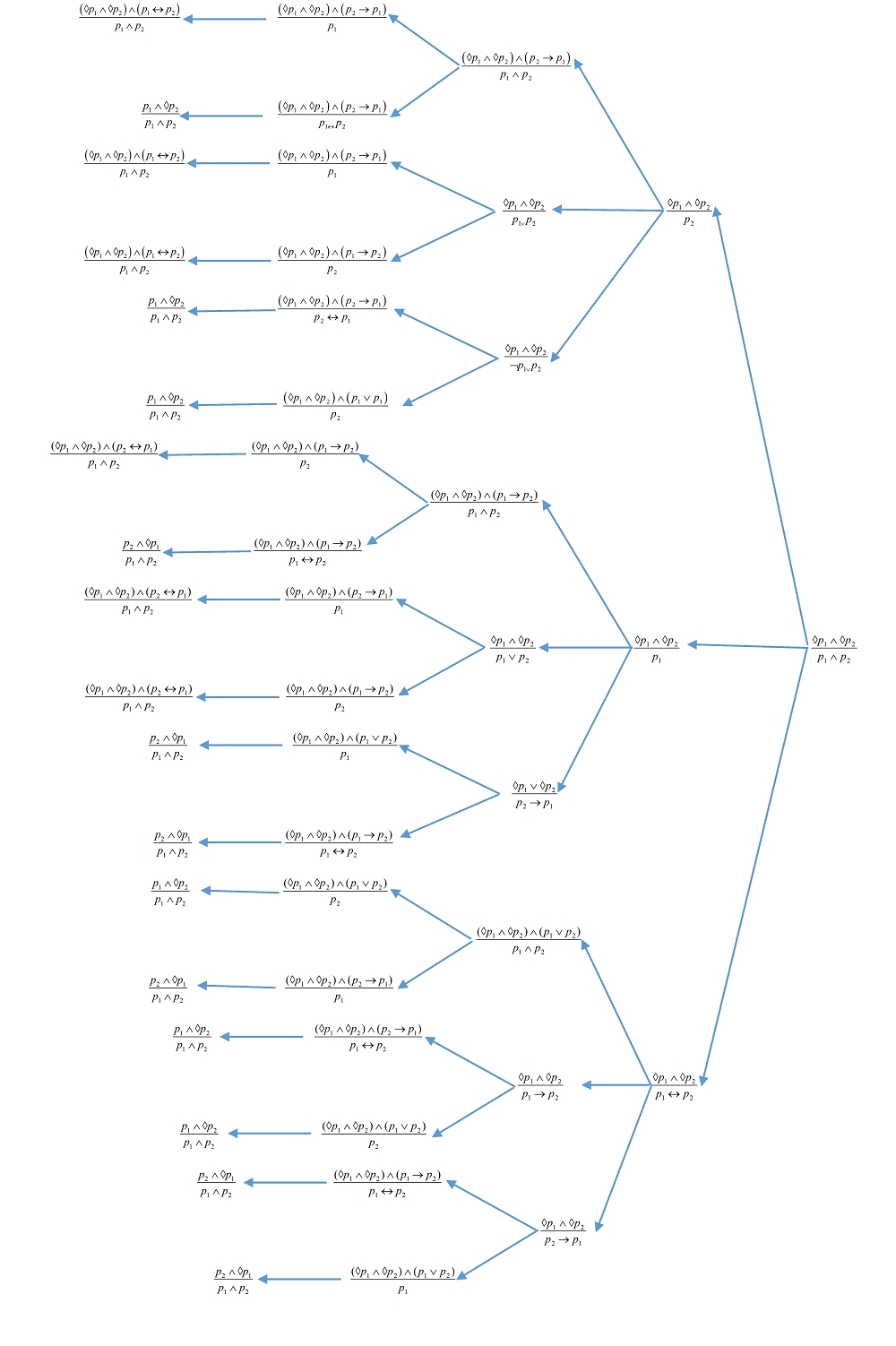}
    \caption{\label{fig:Doc2}}
\end{figure}
\begin{figure}
  \centering
    \includegraphics[scale=.6]{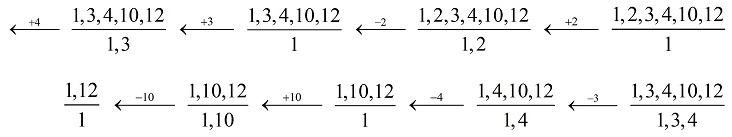}
  \caption{\label{fig:ab}}
\end{figure}
\newpage
\begin{Ex}
{\tiny
$S(\frac{1,2,3,4,10,12}{1})\stackrel{+2}{=} S(\frac{1,2,3,4,10,12}{1,2})\cup S(\frac{1,2}{1})$\\

$\stackrel{-2}{=}S(\frac{1,3,4,10,12}{1})\cup S(\frac{1,2,3,4,10,12}{1,2},\frac{1,2,3,4,10,12}{1,3,4,10,12})\cup S(\frac{1,2}{1})$\\

$\stackrel{+3}{=}S(\frac{1,3,4,10,12}{1,3})\cup S(\frac{1,3}{1})\cup S(\frac{1,2,3,4,10,12}{1,2,3},\frac{1,2,3,4,10,12}{1,3,4,10,12})\cup S(\frac{1,2,3}{1,2},\frac{1,2,3,4,10,12}{1,3,4,10,12})\cup S(\frac{1,2}{1})$\\

$\stackrel{s}{=}S(\frac{1,3,4,10,12}{1,3})\cup S(\frac{1,3}{1})\cup S(\frac{1,2,3,4,10,12}{1,2,3},\frac{1,2,3,4,10,12}{1,3,4,10,12})\cup S(\frac{1,2,3}{1,2},\frac{1,2,3}{1,3})\cup S(\frac{1,2}{1})$\\

$\stackrel{+4}{=}S(\frac{1,3,4,10,12}{1,3,4})\cup S(\frac{1,3,4}{1,3})\cup S(\frac{1,3}{1})\cup S(\frac{1,2,3,4,10,12}{1,2,3,4},\frac{1,2,3,4,10,12}{1,3,4,10,12})\cup S(\frac{1,2,3,4}{1,2,3},\frac{1,2,3,4,10,12}{1,3,4,10,12})\cup S(\frac{1,2,3}{1,2},\frac{1,2,3}{1,3})\cup S(\frac{1,2}{1})$\\

$\stackrel{s}{=}S(\frac{1,3,4,10,12}{1,3,4})\cup S(\frac{1,3,4}{1,3})\cup S(\frac{1,3}{1})\cup S(\frac{1,2,3,4,10,12}{1,2,3,4},\frac{1,2,3,4,10,12}{1,3,4,10,12})\cup S(\frac{1,2,3,4}{1,2,3},\frac{1,2,3,4}{1,3,4})\cup S(\frac{1,2,3}{1,2},\frac{1,2,3}{1,3})\cup S(\frac{1,2}{1})$

$\stackrel{-3}{=}S(\frac{1,4,10,12}{1,4})\cup S(\frac{1,3,4,10,12}{1,3,4},\frac{1,3,4,10,12}{1,4,10,12})\cup S(\frac{1,4}{1})\cup S(\frac{1,3,4}{1,3},\frac{1,3,4}{1,4})\cup S(\frac{1,3}{1})\cup S(\frac{1,2,4,10,12}{1,2,4},\frac{1,2,4,10,12}{1,4,10,12})\cup \\
S(\frac{1,2,3,4,10,12}{1,2,3,4},\frac{1,2,3,4,10,12}{1,3,4,10,12},\frac{1,2,3,4,10,12}{1,2,4,10,12})\cup S(\frac{1,2,4}{1,2},\frac{1,2,4}{1,4})\cup S(\frac{1,2,3,4}{1,2,3},\frac{1,2,3,4}{1,3,4},\frac{1,2,3,4}{1,2,4})\cup S(\frac{1,2}{1,2},\frac{1,2}{1})\cup S(\frac{1,2,3}{1,2},\frac{1,2,3}{1,3},\frac{1,2,3}{1,2})\cup S(\frac{1,2}{1})$
\vspace{1cm}

$\stackrel{s}{=}S(\frac{1,4,10,12}{1,4})\cup S(\frac{1,3,4,10,12}{1,3,4},\frac{1,3,4,10,12}{1,4,10,12})\cup S(\frac{1,4}{1})\cup S(\frac{1,3,4}{1,3},\frac{1,3,4}{1,4})\cup S(\frac{1,3}{1})\cup S(\frac{1,2,4,10,12}{1,2,4},\frac{1,2,4,10,12}{1,4,10,12})\cup \\ S(\frac{1,2,3,4,10,12}{1,2,3,4},\frac{1,2,3,4,10,12}{1,3,4,10,12},\frac{1,2,3,4,10,12}{1,2,4,10,12})\cup S(\frac{1,2,4}{1,2},\frac{1,2,4}{1,4})\cup S(\frac{1,2,3,4}{1,2,3},\frac{1,2,3,4}{1,3,4},\frac{1,2,3,4}{1,2,4})\cup S(\frac{1,2,3}{1,2},\frac{1,2,3}{1,3})\cup S(\frac{1,2}{1})$
\vspace{1cm}

$\stackrel{-4}{=}S(\frac{1,10,12}{1})\cup S(\frac{1,4,10,12}{1,4},\frac{1,4,10,12}{1,10,12})\cup S(\frac{1,3,10,12}{1,3},\frac{1,3,10,12}{1,10,12})\cup S(\frac{1,3,4,10,12}{1,3,4},\frac{1,3,4,10,12}{1,4,10,12},\frac{1,3,4,10,12}{1,3,10,12})\cup S(\frac{1,4}{1})\cup S(\frac{1,3}{1,3},\frac{1,3}{1})\cup S(\frac{1,3,4}{1,3},\frac{1,3,4}{1,4})\cup S(\frac{1,3}{1})\cup S(\frac{1,2,10,12}{1,2},\frac{1,2,10,12}{1,10,12})\cup S(\frac{1,2,4,10,12}{1,2,4},\frac{1,2,4,10,12}{1,4,10,12},\frac{1,2,4,10,12}{1,2,10,12})\cup S(\frac{1,2,3,10,12}{1,2,3},\frac{1,2,3,10,12}{1,3,10,12},\frac{1,2,3,10,12}{1,2,10,12})\cup S(\frac{1,2,3,4,10,12}{1,2,3,4},\frac{1,2,3,4,10,12}{1,3,4,10,12},\frac{1,2,3,4,10,12}{1,2,4,10,12},\frac{1,2,3,4,10,12}{1,2,3,10,12})\cup S(\frac{1,2}{1,2},\frac{1,2}{1})\cup S(\frac{1,2,4}{1,2},\frac{1,2,4}{1,4},\frac{1,2,4}{1,2})\cup S(\frac{1,2,3}{1,2,3},\frac{1,2,3}{1,3},\frac{1,2,3}{1,2})\cup\\ S(\frac{1,2,3,4}{1,2,3},\frac{1,2,3,4}{1,3,4},\frac{1,2,3,4}{1,2,4},\frac{1,2,3,4}{1,2,3})\cup S(\frac{1,2,3}{1,2},\frac{1,2,3}{1,3}) \cup S(\frac{1,2}{1})$\\

$\stackrel{s}{=}S(\frac{1,10,12}{1})\cup S(\frac{1,4,10,12}{1,4},\frac{1,4,10,12}{1,10,12})\cup S(\frac{1,2,3,10,12}{1,3},\frac{1,2,3,10,12}{1,10,12})\cup S(\frac{1,2,3,4,10,12}{1,3,4},\frac{1,2,3,4,10,12}{1,4,10,12},\frac{1,2,3,4,10,12}{1,2,3,10,12})\cup S(\frac{1,4}{1})\cup\\ S(\frac{1,3,4}{1,3},\frac{1,3,4}{1,4})\cup S(\frac{1,3}{1})\cup S(\frac{1,2,10,12}{1,2},\frac{1,2,10,12}{1,10,12})\cup S(\frac{1,2,4,10,12}{1,2,4},\frac{1,2,4,10,12}{1,4,10,12},\frac{1,2,4,10,12}{1,2,10,12})\cup S(\frac{1,2,3,10,12}{1,2,3},\frac{1,2,3,10,12}{1,3,10,12},\frac{1,2,3,10,12}{1,2,10,12})\cup S(\frac{1,2,3,4,10,12}{1,2,3,4},\frac{1,2,3,4,10,12}{1,3,4,10,12},\frac{1,2,3,4,10,12}{1,2,4,10,12},\frac{1,2,3,4,10,12}{1,2,3,10,12})\cup S(\frac{1,2,4}{1,2},\frac{1,2,4}{1,4})\cup S(\frac{1,2,3,4}{1,2,3},\frac{1,2,3,4}{1,3,4},\frac{1,2,3,4}{1,2,4})\cup S(\frac{1,2,3}{1,2},\frac{1,2,3}{1,3}) \cup S(\frac{1,2}{1})$\\

$\stackrel{+10}{=}S(\frac{1,10,12}{1,10})\cup  S(\frac{1,10}{1})\cup S(\frac{1,4,10,12}{1,4,10},\frac{1,4,10,12}{1,10,12})\cup S(\frac{1,4,10}{1,4},\frac{1,4,10,12}{1,10,12})\cup S(\frac{1,2,3,10,12}{1,3,10},\frac{1,2,3,10,12}{1,10,12})\cup S(\frac{1,3,10}{1,3},\frac{1,2,3,10,12}{1,10,12})\cup S(\frac{1,2,3,4,10,12}{1,3,4,10},\frac{1,2,3,4,10,12}{1,4,10,12},\frac{1,2,3,4,10,12}{1,2,3,10,12})\cup S(\frac{1,3,4,10}{1,3,4},\frac{1,2,3,4,10,12}{1,4,10,12},\frac{1,2,3,4,10,12}{1,2,3,10,12})\cup
S(\frac{1,4}{1})\cup S(\frac{1,3,4}{1,3},\frac{1,3,4}{1,4})\cup S(\frac{1,3}{1})\cup\\
 S(\frac{1,2,10,12}{1,2,10},\frac{1,2,10,12}{1,10,12})\cup S(\frac{1,2,10}{1,2},\frac{1,2,10,12}{1,10,12})\cup S(\frac{1,2,4,10,12}{1,2,4,10},\frac{1,2,4,10,12}{1,4,10,12},\frac{1,2,4,10,12}{1,2,10,12})\cup S(\frac{1,2,4,10}{1,2,4},\frac{1,2,4,10,12}{1,4,10,12},\frac{1,2,4,10,12}{1,2,10,12})\cup \\ S(\frac{1,2,3,10,12}{1,2,3,10},\frac{1,2,3,10,12}{1,3,10,12},\frac{1,2,3,10,12}{1,2,10,12})\cup S(\frac{1,2,3,10}{1,2,3},\frac{1,2,3,10,12}{1,3,10,12},\frac{1,2,3,10,12}{1,2,10,12})\cup S(\frac{1,2,3,4,10,12}{1,2,3,4,10},\frac{1,2,3,4,10,12}{1,3,4,10,12},\frac{1,2,3,4,10,12}{1,2,4,10,12},\frac{1,2,3,4,10,12}{1,2,3,10,12})\cup$
$S(\frac{1,2,3,4,10}{1,2,3,4},\frac{1,2,3,4,10,12}{1,3,4,10,12},\frac{1,2,3,4,10,12}{1,2,4,10,12},\frac{1,2,3,4,10,12}{1,2,3,10,12})\cup
S(\frac{1,2,4}{1,2},\frac{1,2,4}{1,4})\cup S(\frac{1,2,3,4}{1,2,3},\frac{1,2,3,4}{1,3,4},\frac{1,2,3,4}{1,2,4})\cup S(\frac{1,2,3}{1,2}) \cup S(\frac{1,2}{1})$\\

$\stackrel{s}{=}S(\frac{1,10,12}{1,10})\cup S(\frac{1,10}{1})\cup S(\frac{1,4,10,12}{1,4,10},\frac{1,4,10,12}{1,10,12})\cup S(\frac{1,4,10}{1,4},\frac{1,4,10}{1,10})\cup S(\frac{1,2,3,10,12}{1,3,10},\frac{1,2,3,10,12}{1,10,12})\cup S(\frac{1,3,10}{1,3},\frac{1,3,10}{1,10})\cup\\
 S(\frac{1,2,3,4,10,12}{1,3,4,10},\frac{1,2,3,4,10,12}{1,4,10,12},\frac{1,2,3,4,10,12}{1,2,3,10,12})
\cup S(\frac{1,3,4,10}{1,3,4},\frac{1,3,4,10}{1,4,10},\frac{1,3,4,10}{1,3,10})\cup S(\frac{1,4}{1})
\cup S(\frac{1,3,4}{1,3},\frac{1,3,4}{1,4})\cup S(\frac{1,3}{1})\cup S(\frac{1,2,10,12}{1,2,10},\frac{1,2,10,12}{1,10,12})
\cup S(\frac{1,2,10}{1,2},\frac{1,2,10}{1,10})\cup S(\frac{1,2,4,10,12}{1,2,4,10},\frac{1,2,4,10,12}{1,4,10,12},\frac{1,2,4,10,12}{1,2,10,12})
\cup S(\frac{1,2,4,10}{1,2,4},\frac{1,2,4,10}{1,4,10},\frac{1,2,4,10}{1,2,10})\cup S(\frac{1,2,3,10,12}{1,2,3,10},\frac{1,2,3,10,12}{1,3,10,12},\frac{1,2,3,10,12}{1,2,10,12})\cup \\ S(\frac{1,2,3,10}{1,2,3},\frac{1,2,3,10}{1,3,10},\frac{1,2,3,10}{1,2,10})\cup
S(\frac{1,2,3,4,10,12}{1,2,3,4,10},\frac{1,2,3,4,10,12}{1,3,4,10,12},\frac{1,2,3,4,10,12}{1,2,4,10,12},\frac{1,2,3,4,10,12}{1,2,3,10,12})\cup
S(\frac{1,2,3,4,10}{1,2,3,4},\frac{1,2,3,4,10}{1,3,4,10},\frac{1,2,3,4,10}{1,2,4,10},\frac{1,2,3,4,10}{1,2,3,10})\cup
S(\frac{1,2,4}{1,2},\frac{1,2,4}{1,4})\cup S(\frac{1,2,3,4}{1,2,3},\frac{1,2,3,4}{1,3,4},\frac{1,2,3,4}{1,2,4})\cup S(\frac{1,2,3}{1,2}) \cup S(\frac{1,2}{1})$\\

$\stackrel{-10}{=}S(\frac{1,12}{1})\cup S(\frac{1,10,12}{1,10},\frac{1,10,12}{1,12})\cup S(\frac{1,10}{1})\cup S(\frac{1,4,12}{1,4},\frac{1,4,12}{1,12})\cup S(\frac{1,4,10,12}{1,4,10},\frac{1,4,10,12}{1,10,12},\frac{1,4,10,12}{1,4,12})\cup S(\frac{1,4,10}{1,4},\frac{1,4,10}{1,10})\cup \\
S(\frac{1,2,3,12}{1,3},\frac{1,2,3,12}{1,12})\cup S(\frac{1,2,3,10,12}{1,3,10},\frac{1,2,3,10,12}{1,10,12},\frac{1,2,3,10,12}{1,2,3,12})\cup S(\frac{1,3,10}{1,3},\frac{1,3,10}{1,10})\cup S(\frac{1,2,3,4,12}{1,3,4},\frac{1,2,3,4,12}{1,4,12},\frac{1,2,3,4,12}{1,2,3,12})\cup\\ S(\frac{1,2,3,4,10,12}{1,3,4,10},\frac{1,2,3,4,10,12}{1,4,10,12},\frac{1,2,3,4,10,12}{1,2,3,10,12},\frac{1,2,3,4,10,12}{1,2,3,4,12})\cup S(\frac{1,3,4,10}{1,3,4},\frac{1,3,4,10}{1,4,10},\frac{1,3,4,10}{1,3,10})\cup S(\frac{1,4}{1})\cup S(\frac{1,3,4}{1,3},\frac{1,3,4}{1,4})\cup S(\frac{1,3}{1})\cup \\
S(\frac{1,2,12}{1,2},\frac{1,2,12}{1,12})\cup S(\frac{1,2,10,12}{1,2,10},\frac{1,2,10,12}{1,10,12},\frac{1,2,10,12}{1,2,12})\cup S(\frac{1,2,10}{1,2},\frac{1,2,10}{1,10})
\cup S(\frac{1,2,4,12}{1,2,4},\frac{1,2,4,12}{1,4,12},\frac{1,2,4,12}{1,2,12})\cup \\ S(\frac{1,2,4,10,12}{1,2,4,10},\frac{1,2,4,10,12}{1,4,10,12},\frac{1,2,4,10,12}{1,2,10,12},\frac{1,2,4,10,12}{1,2,4,12})\cup S(\frac{1,2,4,10}{1,2,4},\frac{1,2,4,10}{1,4,10},\frac{1,2,4,10}{1,2,10})\cup S(\frac{1,2,3,12}{1,2,3},\frac{1,2,3,12}{1,3,12},\frac{1,2,3,12}{1,2,12})\cup\\ S(\frac{1,2,3,10,12}{1,2,3,10},\frac{1,2,3,10,12}{1,3,10,12},\frac{1,2,3,10,12}{1,2,10,12},\frac{1,2,3,10,12}{1,2,3,12})\cup S(\frac{1,2,3,10}{1,2,3},\frac{1,2,3,10}{1,3,10},\frac{1,2,3,10}{1,2,10})\cup S(\frac{1,2,3,4,12}{1,2,3,4},\frac{1,2,3,4,12}{1,3,4,12},\frac{1,2,3,4,12}{1,2,4,12},\frac{1,2,3,4,12}{1,2,3,12})\cup\\ S(\frac{1,2,3,4,10,12}{1,2,3,4,10},\frac{1,2,3,4,10,12}{1,3,4,10,12},\frac{1,2,3,4,10,12}{1,2,4,10,12},\frac{1,2,3,4,10,12}{1,2,3,10,12},
\frac{1,2,3,4,10,12}{1,2,3,4,12})\cup
S(\frac{1,2,3,4,10}{1,2,3,4},\frac{1,2,3,4,10}{1,3,4,10},\frac{1,2,3,4,10}{1,2,4,10},\frac{1,2,3,4,10}{1,2,3,10})\cup
S(\frac{1,2,4}{1,2},\frac{1,2,4}{1,4})\cup S(\frac{1,2,3,4}{1,2,3},\frac{1,2,3,4}{1,3,4},\frac{1,2,3,4}{1,2,4})\cup S(\frac{1,2,3}{1,2})\cup S(\frac{1,2}{1})$\\

$\stackrel{s}{=}S(\frac{1,10,12}{1,10})\cup S(\frac{1,10}{1})\cup S(\frac{1,4,10,12}{1,4,10},\frac{1,4,10,12}{1,10,12})\cup S(\frac{1,4,10}{1,4},\frac{1,4,10}{1,10})\cup S(\frac{1,2,3}{1,3})\cup S(\frac{1,2,3,10,12}{1,2,3,10},\frac{1,2,3,10,12}{1,2,10,12})\cup S(\frac{1,2,3,10}{1,3,10},\frac{1,2,3,10}{1,2,10},\frac{1,2,3,10}{1,2,3}) \cup S(\frac{1,2,10,12}{1,2,10},\frac{1,2,10,12}{1,10,12})\cup S(\frac{1,2,10}{1,10},\frac{1,2,10}{1,2})\cup S(\frac{1,3,10}{1,3},\frac{1,3,10}{1,10})\cup S(\frac{1,2,3,4}{1,3,4},\frac{1,2,3,4}{1,2,3})\cup $ $S(\frac{1,2,3,4,10,12}{1,2,3,4,10},\frac{1,2,3,4,10,12}{1,2,4,10,12},\frac{1,2,3,4,10,12}{1,2,3,10,12})\cup\\
 S(\frac{1,2,3,4,10}{1,3,4,10},\frac{1,2,3,4,10}{1,2,4,10},\frac{1,2,3,4,10}{1,2,3,10},\frac{1,2,3,4,10}{1,2,3,4})\cup S(\frac{1,2,4,10,12}{1,2,4,10},\frac{1,2,4,10,12}{1,4,10,12},\frac{1,2,4,10,12}{1,2,10,12})\cup S(\frac{1,2,4,10}{1,4,10},\frac{1,2,4,10}{1,2,10},\frac{1,2,4,10}{1,2,4})\cup \\ S(\frac{1,3,4,10}{1,3,4},\frac{1,3,4,10}{1,4,10},\frac{1,3,4,10}{1,3,10})\cup S(\frac{1,4}{1})\cup S(\frac{1,3,4}{1,3},\frac{1,3,4}{1,4})\cup S(\frac{1,3}{1})\cup S(\frac{1,2,10,12}{1,2,10},\frac{1,2,10,12}{1,10,12})\cup S(\frac{1,2,10}{1,2},\frac{1,2,10}{1,10})\cup\\
 S(\frac{1,2,4,10,12}{1,2,4,10},\frac{1,2,4,10,12}{1,4,10,12},\frac{1,2,4,10,12}{1,2,10,12})\cup S(\frac{1,2,4,10}{1,2,4},\frac{1,2,4,10}{1,4,10},\frac{1,2,4,10}{1,2,10})\cup S(\frac{1,2,3,10,12}{1,2,3,10},\frac{1,2,3,10,12}{1,3,10,12},\frac{1,2,3,10,12}{1,2,10,12})\cup S(\frac{1,2,3,10}{1,2,3},\frac{1,2,3,10}{1,3,10},\frac{1,2,3,10}{1,2,10})\cup\\ S(\frac{1,2,3,4,10,12}{1,2,3,4,10},\frac{1,2,3,4,10,12}{1,3,4,10,12},\frac{1,2,3,4,10,12}{1,2,4,10,12},\frac{1,2,3,4,10,12}{1,2,3,10,12})\cup
S(\frac{1,2,3,4,10}{1,2,3,4},\frac{1,2,3,4,10}{1,3,4,10},\frac{1,2,3,4,10}{1,2,4,10},\frac{1,2,3,4,10}{1,2,3,10})\cup
S(\frac{1,2,4}{1,2},\frac{1,2,4}{1,4})\cup\\
S(\frac{1,2,3,4}{1,2,3},\frac{1,2,3,4}{1,3,4},\frac{1,2,3,4}{1,2,4})\cup S(\frac{1,2,3}{1,2}) \cup S(\frac{1,2}{1})$}\\

 In the last equation note $\frac{1,12}{1},\frac{1,4,12}{1,4},\frac{1,2,3,12}{1,2,3},\frac{1,2,3,4,12}{1,2,3,4}$ are admissible, otherwise by Theorem $3.6$ and Figure $(2)$ there is $W\subseteq\{1,3,4,12\}$ such that $12\in W$ and ${\mathcal M}(\Theta_2,W), \phi_{12}\vDash \phi_{12}$ then ${\mathcal M}(\Theta_2,W), \phi_{12}\vDash \diamond p_2$ which is false. Then by Theorem $4.1$, parts $(4)$ and $(5)$  \\

 $S(\frac{1,12}{1})=\varnothing$\\

$S(\frac{1,4,12}{1,4},\frac{1,4,12}{1,12})=S(\frac{1,4,12}{1,4},\frac{1,4,12}{1})=S(\frac{1,4,12}{1,4})=\varnothing$\\

$S(\frac{1,2,12}{1,2},\frac{1,2,12}{1,12})=S(\frac{1,2,12}{1,2},\frac{1,2,12}{1})=S(\frac{1,2,12}{1,2})=\varnothing$\\

$S(\frac{1,2,4,12}{1,2,4},\frac{1,2,4,12}{1,4,12},\frac{1,2,4,12}{1,2,12})=S(\frac{1,2,4,12}{1,2,4},\frac{1,2,4,12}{1,4},\frac{1,2,4,12}{1,2})=
S(\frac{1,2,4,12}{1,2,4})=\varnothing$\\

$S(\frac{1,2,3,12}{1,2,3},\frac{1,2,3,12}{1,3,12},\frac{1,2,3,12}{1,2,12})=S(\frac{1,2,3,12}{1,2,3},\frac{1,2,3,12}{1,3},\frac{1,2,3,12}{1,2})=
S(\frac{1,2,3,12}{1,2,3})=\varnothing$
\vspace{1cm}

$S(\frac{1,2,3,4,12}{1,2,3,4},\frac{1,2,3,4,12}{1,3,4,12},\frac{1,2,3,4,12}{1,2,4,12},\frac{1,2,3,4,12}{1,2,3,12})=
S(\frac{1,2,3,4,12}{1,2,3,4},\frac{1,2,3,4,12}{1,3,4},\frac{1,2,3,4,12}{1,2,4},\frac{1,2,3,4,12}{1,2,3})=S(\frac{1,2,3,4,12}{1,2,3,4})=\emptyset$
\vspace{1cm}

$S(\frac{1,10,12}{1,10},\frac{1,10,12}{1,12})=S(\frac{1,10,12}{1,10},\frac{1,10,12}{1})=S(\frac{1,10,12}{1,10})$,
\vspace{1cm}

$S(\frac{1,4,10,12}{1,4,10},\frac{1,4,10,12}{1,10,12},\frac{1,4,10,12}{1,4,12})=
S(\frac{1,4,10,12}{1,4,10},\frac{1,4,10,12}{1,10,12},\frac{1,4,10,12}{1,4})=S(\frac{1,4,10,12}{1,4,10},\frac{1,4,10,12}{1,10,12}),$
\vspace{1cm}

$S(\frac{1,2,3,12}{1,3},\frac{1,2,3,12}{1,12})=S(\frac{1,2,3}{1,3},\frac{1,2,3}{1})=S(\frac{1,2,3}{1,3}),$
\vspace{1cm}

$S(\frac{1,2,3,10,12}{1,3,10},\frac{1,2,3,10,12}{1,10,12},\frac{1,2,3,10,12}{1,2,3,12})\stackrel{+2}{=}
S(\frac{1,2,3,10,12}{1,2,3,10},\frac{1,2,3,10,12}{1,2,10,12},\frac{1,2,3,10,12}{1,2,3,12})\cup S(\frac{1,2,3,10}{1,3,10},\frac{1,2,3,10,12}{1,2,10,12},\frac{1,2,3,10,12}{1,2,3,12}) \cup S(\frac{1,2,3,10,12}{1,2,3,10},\frac{1,2,10,12}{1,10,12},\frac{1,2,3,10,12}{1,2,3,12})\cup S(\frac{1,2,3,10}{1,3,10},\frac{1,2,10,12}{1,10,12},\frac{1,2,3,10,12}{1,2,3,12})=
S(\frac{1,2,3,10,12}{1,2,3,10},\frac{1,2,3,10,12}{1,2,10,12},\frac{1,2,3,10,12}{1,2,3})\cup S(\frac{1,2,3,10}{1,3,10},\frac{1,2,3,10}{1,2,10},\frac{1,2,3,10}{1,2,3}) \cup S(\frac{1,2,10,12}{1,2,10},\frac{1,2,10,12}{1,10,12},\frac{1,2,10,12}{1,2})\cup S(\frac{1,2,10}{1,10},\frac{1,2,10}{1,10},\frac{1,2,10}{1,2})=S(\frac{1,2,3,10,12}{1,2,3,10},\frac{1,2,3,10,12}{1,2,10,12})\cup S(\frac{1,2,3,10}{1,3,10},\frac{1,2,3,10}{1,2,10},\frac{1,2,3,10}{1,2,3}) \cup S(\frac{1,2,10,12}{1,2,10},\frac{1,2,10,12}{1,10,12})\cup S(\frac{1,2,10}{1,10},\frac{1,2,10}{1,2})$
\vspace{1cm}

$S(\frac{1,2,3,4,12}{1,3,4},\frac{1,2,3,4,12}{1,4,12},\frac{1,2,3,4,12}{1,2,3,12})=S(\frac{1,2,3,4}{1,3,4},\frac{1,2,3,4}{1,4},\frac{1,2,3,4}{1,2,3})=
S(\frac{1,2,3,4}{1,3,4},\frac{1,2,3,4}{1,2,3})$,
\vspace{1cm}

 $S(\frac{1,2,10,12}{1,2,10},\frac{1,2,10,12}{1,10,12},
\frac{1,2,10,12}{1,2,12})=S(\frac{1,2,10,12}{1,2,10},\frac{1,2,10,12}{1,10,12},
\frac{1,2,10,12}{1,2})=S(\frac{1,2,10,12}{1,2,10},\frac{1,2,10,12}{1,10,12})$,
\vspace{1cm}

$S(\frac{1,2,4,10,12}{1,2,4,10},\frac{1,2,4,10,12}{1,4,10,12},\frac{1,2,4,10,12}{1,2,10,12},\frac{1,2,4,10,12}{1,2,4,12})=
S(\frac{1,2,4,10,12}{1,2,4,10},\frac{1,2,4,10,12}{1,4,10,12},\frac{1,2,4,10,12}{1,2,10,12},\frac{1,2,4,10,12}{1,2,4})=\\
S(\frac{1,2,4,10,12}{1,2,4,10},\frac{1,2,4,10,12}{1,4,10,12},\frac{1,2,4,10,12}{1,2,10,12})$,
\vspace{1cm}

$S(\frac{1,2,3,10,12}{1,2,3,10},\frac{1,2,3,10,12}{1,3,10,12},\frac{1,2,3,10,12}{1,2,10,12},\frac{1,2,3,10,12}{1,2,3,12})=
S(\frac{1,2,3,10,12}{1,2,3,10},\frac{1,2,3,10,12}{1,3,10,12},\frac{1,2,3,10,12}{1,2,10,12},\frac{1,2,3,10,12}{1,2,3})=\\
S(\frac{1,2,3,10,12}{1,2,3,10},
\frac{1,2,3,10,12}{1,3,10,12},\frac{1,2,3,10,12}{1,2,10,12})$,
\vspace{1cm}

$S(\frac{1,2,3,4,10,12}{1,2,3,4,10},\frac{1,2,3,4,10,12}{1,3,4,10,12},\frac{1,2,3,4,10,12}{1,2,4,10,12},\frac{1,2,3,4,10,12}{1,2,3,10,12},
\frac{1,2,3,4,10,12}{1,2,3,4,12})=\\
S(\frac{1,2,3,4,10,12}{1,2,3,4,10},\frac{1,2,3,4,10,12}{1,3,4,10,12},\frac{1,2,3,4,10,12}{1,2,4,10,12},\frac{1,2,3,4,10,12}{1,2,3,10,12},
\frac{1,2,3,4,10,12}{1,2,3,4})=\\
S(\frac{1,2,3,4,10,12}{1,2,3,4,10},\frac{1,2,3,4,10,12}{1,3,4,10,12},\frac{1,2,3,4,10,12}{1,2,4,10,12},
\frac{1,2,3,4,10,12}{1,2,3,10,12})$,
\vspace{1cm}

$S(\frac{1,2,3,4,10,12}{1,3,4,10},\frac{1,2,3,4,10,12}{1,4,10,12},\frac{1,2,3,4,10,12}{1,2,3,10,12},\frac{1,2,3,4,10,12}{1,2,3,4,12})\stackrel{+2}{=}
S(\frac{1,2,3,4,10,12}{1,2,3,4,10},\frac{1,2,3,4,10,12}{1,2,4,10,12},\frac{1,2,3,4,10,12}{1,2,3,10,12},\frac{1,2,3,4,10,12}{1,2,3,4,12})\cup S(\frac{1,2,3,4,10}{1,3,4,10},\frac{1,2,3,4,10,12}{1,2,4,10,12},\frac{1,2,3,4,10,12}{1,2,3,10,12},\frac{1,2,3,4,10,12}{1,2,3,4,12})\cup S(\frac{1,2,3,4,10,12}{1,2,3,4,10},\frac{1,2,4,10,12}{1,4,10,12},\frac{1,2,3,4,10,12}{1,2,3,10,12},\frac{1,2,3,4,10,12}{1,2,3,4,12})\cup\\
S(\frac{1,2,3,4,10}{1,3,4,10},\frac{1,2,4,10,12}{1,4,10,12},\frac{1,2,3,4,10,12}{1,2,3,10,12},\frac{1,2,3,4,10,12}{1,2,3,4,12})=
S(\frac{1,2,3,4,10,12}{1,2,3,4,10},\frac{1,2,3,4,10,12}{1,2,4,10,12},\frac{1,2,3,4,10,12}{1,2,3,10,12},\frac{1,2,3,4,10,12}{1,2,3,4})\cup\\
S(\frac{1,2,3,4,10}{1,3,4,10},\frac{1,2,3,4,10}{1,2,4,10},\frac{1,2,3,4,10}{1,2,3,10},\frac{1,2,3,4,10}{1,2,3,4})\cup S(\frac{1,2,4,10,12}{1,2,4,10},\frac{1,2,4,10,12}{1,4,10,12},\frac{1,2,4,10,12}{1,2,10,12},\frac{1,2,4,10,12}{1,2,4})\cup\\
S(\frac{1,2,4,10}{1,4,10},\frac{1,2,4,10}{1,4,10},\frac{1,2,4,10}{1,2,10},\frac{1,2,4,10}{1,2,4})=
S(\frac{1,2,3,4,10,12}{1,2,3,4,10},\frac{1,2,3,4,10,12}{1,2,4,10,12},\frac{1,2,3,4,10,12}{1,2,3,10,12})\cup
\\ S(\frac{1,2,3,4,10}{1,3,4,10},\frac{1,2,3,4,10}{1,2,4,10},\frac{1,2,3,4,10}{1,2,3,10},\frac{1,2,3,4,10}{1,2,3,4})\cup S(\frac{1,2,4,10,12}{1,2,4,10},\frac{1,2,4,10,12}{1,4,10,12},\frac{1,2,4,10,12}{1,2,10,12})\cup S(\frac{1,2,4,10}{1,4,10},\frac{1,2,4,10}{1,2,10},\frac{1,2,4,10}{1,2,4}). $
\end{Ex}
\section{Conclusion and Discussion}
 In this paper, we investigated some relations between the method in \cite{Sergey1}, \cite{18}, based on the reduced normal form rules in Theorem $3.5$, and sets of substitutions which reject them in section $4$. We also generalized the method for one rule to inadmissibility of a set of rules. We did some case studies for the cases $2$ and $3$ variables. The case studies show complexity of the problem. The decomposition of sets of substitutions rejecting sets of rules to its components algorithmically is done in section $4$ and calculating at least one member of each components is leaved. Partially the problem is solved for some cases in the paper.

\end{document}